\documentclass[12pt]{article}

\usepackage{amsmath}
\usepackage{amsfonts}
\usepackage{amssymb}
\usepackage{amsthm}

\usepackage{graphicx}
\usepackage[shortlabels]{enumitem}
\usepackage{float}
\usepackage[bottom=2cm, right=2cm, left=2cm, top=2cm]{geometry}

\usepackage{thmtools}
\usepackage{thm-restate}
\usepackage{hyperref}
\hypersetup{colorlinks=true, unicode=true, linkcolor=[rgb]{0.10,0.05,0.67}, citecolor=[rgb]{0.10,0.05,0.67}, filecolor=[rgb]{0.10,0.05,0.67}, urlcolor=[rgb]{0.10,0.05,0.67}}
\usepackage[noabbrev,capitalize]{cleveref}

\usepackage{tikz}
\usetikzlibrary{arrows}
\usetikzlibrary{decorations.markings}

\tikzset{->-/.style={decoration={
  markings,
  mark=at position .5 with {\arrow{>}}},postaction={decorate}}}
\tikzset{-->-/.style={decoration={
  markings,
  mark=at position .65 with {\arrow{>}}},postaction={decorate}}}
  
\tikzset{played/.style={>=stealth',->-}}
\tikzset{makerlastplayed/.style={>=stealth',->-,blue,very thick}}
\tikzset{breakerlastplayed/.style={dashed, >=stealth',->-,red,thick}}
\tikzset{breakerlastplayed2/.style={dotted,>=stealth',->-,red,thick}}
\tikzset{available/.style={>=stealth',->-,lightgray}}

\newcommand{\R}{\mathbb{R}}
\newcommand{\Z}{\mathbb{Z}}
\newcommand{\floor}[1]{\left\lfloor #1 \right\rfloor}

\newtheorem{theorem}{Theorem}[section]
\newtheorem{proposition}[theorem]{Proposition}
\newtheorem{definition}[theorem]{Definition}
\newtheorem{conjecture}[theorem]{Conjecture}
\numberwithin{equation}{section}

\newcommand{\Gb}{G_b}
\newcommand{\gb}{g_b}
\newcommand{\gbinv}{g^{-1}_b}
\newcommand{\gDomUB}{\floor{\gbinv(1/2)}}
\newcommand{\diffContribution}{\tau}
\newcommand{\costf}{Q}
\newcommand{\epb}[1]{\varepsilon_{b,#1}}
\newcommand{\taub}[1]{\tau_{b,#1}}
\newcommand{\pb}[1]{\delta_{b,#1}}
\newcommand{\smb}[1]{s_{b,#1}}
\newcommand{\myfloor}[1]{\underline{#1}}
\newcommand{\CB}{C_B}
\newcommand{\Cb}{C_b}
\newcommand{\minfrac}[1]{\phi_{b,#1}}

\title{An Improved Upper Bound on the Threshold Bias of the Oriented-cycle Game}

\author{
     Anita Liebenau \thanks{School of Mathematics and Statistics, UNSW Sydney, NSW 2052, Australia. 
  Email: \texttt{a.liebenau@unsw.edu.au}}
\and Abdallah Saffidine \thanks{Potassco Solutions, Potsdam, Germany. 
  Email: \texttt{abdallah.saffidine@gmail.com}}
\and Jeffrey Yang \thanks{School of Mathematics and Statistics, UNSW Sydney, NSW 2052, Australia. Email: \texttt{j.z.yang@student.unsw.edu.au}}
}

\begin{document}
\maketitle

\begin{abstract}
We study the $b$-biased Oriented-cycle game where two players, OMaker and OBreaker, take turns directing the edges of $K_n$ (the complete graph on $n$ vertices). In each round, OMaker directs one previously undirected edge followed by OBreaker directing between one and $b$ previously undirected edges. The game ends once all edges have been directed, and OMaker wins if and only if the resulting tournament contains a directed cycle. Bollob\'as and Szab\'o asked the following question: what is the largest value of the bias $b$ for which OMaker has a winning strategy? Ben-Eliezer, Krivelevich and Sudakov  proved that OMaker has a winning strategy for $b \leq n/2 - 2$. In the other direction, Clemens and Liebenau proved that OBreaker has a winning strategy for $b \geq 5n/6+2$. Inspired by their approach, we propose a significantly stronger strategy for OBreaker which we prove to be winning for $b \geq 0.7841n + O(1)$.
\end{abstract}

\section{Introduction}
\emph{Orientation games} are played on the board $E(K_n)$ by two players, \emph{OMaker} and \emph{OBreaker}, who take turns directing previously undirected edges until all edges have been directed. The resulting directed graph is called a \emph{tournament}. OMaker wins if and only if the resulting tournament has some prespecified property $\mathcal{P}$. For example, for a given directed graph $H$, the \emph{oriented $H$-game} is defined by the property $\mathcal{P}_H$ of containing a copy of $H$.

As orientation games often favour OMaker, we usually study biased orientation games.  In \textit{$(a:b)$ biased orientation games} played under \textit{monotone rules}, OMaker directs between 1 and $a$ undirected edges on each of their turns, and OBreaker directs between 1 and $b$ undirected edges on each of their turns. This is in contrast to \textit{strict rules} where OMaker directs exactly $a$ undirected edges on each of their turns and OBreaker directs exactly $b$ edges on each of their turns. If there are fewer than the specified number of edges available, then the player directs all remaining edges. 

Observe that $(a:b)$ biased orientation games are closely related to $(a:b)$ biased Maker--Breaker games played on the edge set of the complete graph~$K_n$. In the latter, the players instead claim edges, and Maker wins if the subgraph formed by their edges satisfies some pre-defined property~$\mathcal{P}$. Maker--Breaker games played on~$K_n$ often have a corresponding orientation game variant. For instance, the Maker--Breaker $H$-game is analogous to the aforementioned oriented $H$-game. A significant amount of literature has been devoted to Maker--Breaker games \cite{beck1982remarks,beck1985random,beck1994deterministic,bednarska2000biased,chvatal1978biased,erdos1973combinatorial}.

Despite their similarities, orientation games are \textit{not} Maker--Breaker games: if a player directs the edge $\{u,v\}$ to form the arc $(u,v)$, then this also makes the arc $(v,u)$ unavailable. Also, in Maker--Breaker games, claiming more edges can only help a player whereas in orientation games, the edges directed by a player may be exploited by the other player. As such, while Maker--Breaker games typically assume strict rules, we often study orientation games played under monotone rules.

Finally, it may also be noted that both orientation games and Maker--Breaker games belong to the same larger family of \textit{unordered CNF games} as introduced by Ahlroth and Orponen \cite{ahlroth2012unordered}, and later studied by Rahman and Watson \cite{rahman2020complexity, rahman2025tractable, rahman2022erd}. Unordered CNF games, which are the unordered version of QBF, are played on a CNF formula by two players: \textit{Satisfier} and \textit{Falsifier}. Here, the two players take turns assigning truth values to variables of their choice, with Satisfier aiming to make the final formula evaluate to True and Falsifier aiming to make the final formula evaluate to False. Maker--Breaker games can be viewed as unordered CNF games played on monotone formulas, with Maker assuming the role of Falsifier and Breaker assuming the role of Satisfier. Here, if the clauses correspond to Maker's winning sets, then Breaker seeks to satisfy every clause by claiming at least one element from each winning set, while Maker tries to falsify the formula by claiming all elements in some winning set. Orientation games can also be formulated as unordered CNF games, by associating variables with unordered pairs of vertices and truth values with orientations. In this context, clauses express that the desired structural property does not hold. For instance, for the oriented $H$-game, each potential copy of $H$ induces a clause expressing that the copy is blocked (that is, at least one of its edges is oriented incorrectly from OMaker's perspective). Similar to Maker--Breaker games, OMaker assumes the role of Falsifier while OBreaker assumes the role of Satisfier.

Of particular interest are $(1:b)$ biased orientation games, called \textit{$b$-biased orientation games}. Here, OMaker directs one undirected edge on each of their turns. On the other hand, OBreaker directs between 1 and $b$ undirected edges on their turn. Observe that increasing $b$ can only benefit OBreaker as they can opt to direct fewer than $b$ edges per round. That is, these games are \textit{bias monotone}. Therefore, for each $K_n$ and non-degenerate property $\mathcal{P}$, there exists some \textit{threshold bias} $t(n, \mathcal{P})$ for which OMaker wins if and only if $b \leq t(n, \mathcal{P})$.  It is not clear whether strict $b$-biased orientation games are bias monotone. Thus, we denote by $t^+(n,\mathcal{P})$ the maximum value of $b$ for which OMaker wins the strict $b$-biased orientation game. It is clear that $t(n,\mathcal{C}) \leq t^+(n,\mathcal{C})$ but not much else is known in general about the relationship between the two quantities.

We study the \textit{Oriented-cycle game} which is an orientation game defined by the property $\mathcal{C}$ of containing a directed cycle. The strict $b$-biased Oriented-cycle game was introduced by Alon and later studied by Bollobás and Szabó \cite{BOLLOBAS199855}. The strict and monotone variants of the $b$-biased Oriented-cycle game have since been studied in various other works \cite{BEDNARSKA2005271,BENELIEZER20121732,CLEMENS201721}. We briefly describe some known results.

Regarding the lower bound of the threshold bias, Alon showed that $t^+(n,\mathcal{C}) \geq \floor{n/4}$ (unpublished but mentioned in \cite{BOLLOBAS199855}). Bollobás and Szabó \cite{BOLLOBAS199855} were able to refine Alon's argument to obtain $t^+(n,\mathcal{C}) \geq \floor{(2-\sqrt{3})n}$. Ben-Eliezer, Krivelevich and Sudakov \cite{BENELIEZER20121732} proved  that OMaker wins for $b \leq n/2 - 2$ (under both monotone and strict rules) via a relatively simple strategy: close a directed cycle if possible and otherwise, extend the longest directed path. This remains the best known strategy for OMaker at the time of writing. Hence, we know that $t(n,\mathcal{C}) \geq n/2 - 2$.

In the other direction, it is easily observed that OBreaker wins trivially for $b \geq n-2$. Under monotone rules, the trivial strategy for OBreaker goes as follows: whenever OMaker directs some arc $(u,v)$, OBreaker directs all available arcs of the form $(u,w)$ for $w \in V$. The trivial strategy for OBreaker in the strict game is described in \cite{CLEMENS201721} and we omit it here. Bollobás and Szabó \cite{BOLLOBAS199855} conjectured that OBreaker in fact only wins for $b \geq n -2$ so that $t^+(n,\mathcal{C}) = n-3$. This conjecture was refuted by Clemens and Liebenau \cite{CLEMENS201721} who showed that $t^+(n,\mathcal{C}) \leq 19n/20 - 1$ for large enough $n$. Regarding the monotone game, it was also shown in \cite{CLEMENS201721} that OBreaker has a winning strategy for $b \geq 5n/6+2$. Thus, it was established that $t(n,\mathcal{C}) \leq 5n/6 + 1$.

In this paper, we focus on the $b$-biased Oriented-cycle game played under monotone rules. We describe how the strategy in \cite{CLEMENS201721} for OBreaker in the monotone game can be significantly improved to allow OBreaker to ensure that the resulting tournament remains acyclic for notably smaller values of the bias. Concretely, we prove that OBreaker wins for $b \geq 0.7841n + O(1)$ which implies the main theorem of this paper.
\begin{restatable}{thm}{mainresultonthreshold}\label{thm:main-result-on-threshold}
    $t(n,\mathcal{C}) \leq  0.7841n + O(1)$.
\end{restatable}
In fact, we conjecture that $t(n,\mathcal{C}) \geq 3n/4 - O(1)$ and that this upper bound on $t(n,\mathcal{C})$ is somewhat close to optimal. We say more about this in the concluding remarks.

The remainder of the paper is structured as follows. After we introduce some general notation and terminology, we outline the main elements of the strategy for OBreaker in \cite{CLEMENS201721} and describe how it can be improved. We then formally define the directed graph structures which we utilise in our strategy and describe their key properties. Following this, we carefully define several functions and quantities which we use to describe a stronger strategy for OBreaker. Finally, we prove several technical results, which allow us to quantitatively bound the values of the bias for which the new strategy is winning.

\subsection{General Notation and Terminology}
We adopt the notation and terminology in~\cite{CLEMENS201721}. A \textit{directed graph} (digraph) is a pair $D = (V,E)$ consisting of a \textit{vertex set} $V = \{1,2,\dots,n\}$ and an \textit{arc set} $E \subseteq V \times V$. Note that we often identify a digraph $D$ with its arc set $E$. The elements of $V$ are called \textit{vertices} and the elements of $D$ are called \textit{arcs} or \textit{directed edges}. Each arc is an ordered pair of the form $(v,w)$ and its underlying set $\{v,w\}$ is called an \textit{edge}. A \textit{loop} is an arc of the form $(v,v)$. For an arc $(v,w)$, we call $(w,v)$ its \textit{reverse arc}. We define a \textit{simple} digraph as a digraph without loops or pairs of arcs which are the reverse of each other. In this paper, we only consider simple digraphs. Let $\mathcal{L} = \{(v,v) : v\in V\}$ denote the set of loops and let $\overleftarrow{D} := \{(w,v): (v,w) \in D\}$ denote the set of all reverse arcs.  We write $\mathcal{A}(D) := (V\times V)\setminus (D \cup \overleftarrow D \cup \mathcal{L})$ to denote the set of all \textit{available arcs}.

For an arc $e \in D$, we write $e^+$ for its tail and $e^-$ for its head so that $e = (e^+, e^-)$. We write $D^+ := \{e^+ : e \in E\}$ for the set of all tails and $D^- := \{e^- : e \in E\}$ for the set of all heads. We say that $S = (V_S, E_S)$ is a \textit{subdigraph} of $D$ if $V_S \subseteq V$ and $E_S \subseteq E$. For a subset $A \subseteq V$, we write $D[A] := D \cap (A \times A)$ to denote the directed subdigraph of $D$ consisting of arcs spanned by $A$.

A digraph $D = (V,E)$ where $V = \{v_0, v_1, \dots, v_k\}$ (where the $v_i$ are all distinct) is a \emph{directed path} if $E = \{(v_0,v_1), (v_1,v_2), \dots, (v_{k-1},v_k)\}$ and a \emph{directed cycle} if $E= \{(v_0,v_1),\dots,(v_{k-1},v_k)\}\cup\{(v_k,v_0)\}$. A digraph is \textit{acyclic} if it does not contain any directed cycles as a subdigraph. For two disjoint sets $A,B \subseteq V$, we call the pair $(A,B)$ a \textit{uniformly directed biclique} (or UDB for short) if for all $v \in A, w\in B$, we have that $(v,w) \in D$.

An \textit{orientation} of an undirected graph is an assignment of a direction to each edge to produce a directed graph. A \textit{tournament} $T$ is an orientation of the complete graph $K_n$. In our study of the Oriented-cycle game, two players take turns directing the edges of $K_n$ until a tournament $T$ is obtained. Let $D$ denote the directed graph on $n$ vertices formed by the arcs directed so far by the two players. We say that a player \textit{directs the edge} $(v,w)$ if they direct the pair $\{v,w\}$ from $v$ to $w$. This adds $(v,w)$ to $D$ and prevents $(w,v)$ from being added to $D$. For convenience, we define $D + e := D \cup \{e\}$.

\section{Safe Digraphs}
We outline the key elements of the strategy for OBreaker proposed in \cite{CLEMENS201721} to motivate our approach. Their strategy relies on two types of digraph structures: UDBs and $\alpha$-structures. We define $\alpha$-structures formally below but describe their key properties now. An $\alpha$-structure cannot be made cyclic with the addition of a single arc. Given an $\alpha$-structure $H$ and an available arc $e \in \mathcal{A}(H)$, there exists a set of available arcs $F$ for which $(H + e) \cup F$ is also an $\alpha$-structure. The \textit{rank} of an $\alpha$-structure provides an upper bound on the number of additional arcs required to create the new $\alpha$-structure. 

At a high level, OBreaker maintains that the digraph $D$ contains a UDB $(A,B)$ where both $D[V\setminus B]$ and $D[V\setminus A]$ are $\alpha$-structures such that each arc in $D$ starts in $A$ or ends in $B$. Let us call a digraph having these properties a \textit{safe} digraph; we provide a formal definition later and refer the reader to \cref{pic:safe} for an illustration. By maintaining that $D$ is safe, any new edge OMaker directs must lie entirely within $V\setminus A$ or $V \setminus B$. However, since OBreaker maintains that $D[V\setminus B]$ and $D[V\setminus A]$ are $\alpha$-structures, OMaker can never close a cycle on their turn. As we will see shortly, the number of edges OBreaker must direct to recover the safety of the digraph depends on the sizes of $A$ and $B$, and the ranks of the $\alpha$-structures $D[V\setminus B]$ and $D[V\setminus A]$.

The strategy for OBreaker in \cite{CLEMENS201721} consists of three stages. In all three stages, OBreaker maintains that the digraph is safe. Let $(A,B)$ denote the UDB maintained by OBreaker which witnesses that the digraph $D$ is safe. Let $k$ denote the rank of the $\alpha$-structure $D[V\setminus B]$ and $\ell$ denote the rank of the $\alpha$-structure $D[V\setminus A]$. We denote 
\[s := \min(|A| - k, |B| - \ell).\]
In the first stage of their strategy, OBreaker increases $s$ by 1 in each round. It can be shown that for $b \geq 5n/6 + 2$, OBreaker can increase $s$ until $s \geq n - b$. Once this is achieved, OBreaker moves onto the second stage. Here, OBreaker keeps the value of $s$ fixed and only maintains that the digraph is safe. It can be shown that the number of edges OBreaker must direct to re-establish that the digraph is safe after OMaker's move is at most $\max(|B| + k, |A| + \ell)$. However, since $s \geq n - b$, we have that
\[ \max(|B| + k, |A| + \ell) \leq n - \min(|A| - k, |B| - \ell) = n - s \leq b.\]
During the second stage, at least one vertex is added to either $A$ or $B$ in each round. Once $A \cup B = V$, the third stage starts. By this point, OBreaker will have effectively \textit{divided} the board into two sets $A$ and $B$ each of size at most $b$. OBreaker can then easily ensure that both of these smaller digraphs remain acyclic.

To improve upon the result in \cite{CLEMENS201721}, we have the following key insight: if OBreaker can increase $s$ at a faster rate during the first stage, then OBreaker can achieve $s \geq n - b$ for smaller values of $b$. Furthermore, as long as $s \geq n - b$ is achieved while maintaining that the digraph is safe, OBreaker can always go on to complete the remaining stages of the strategy to win the game. Thus, we seek to determine the largest value of $s$ that can be achieved for a given value of the bias $b$.

Having motivated our approach, we now dive into the technical details. We first present the formal definition of $\alpha$-structures as defined in \cite{CLEMENS201721}.
\begin{definition}\label{def:alpha}
Let $V$ be a set of vertices, $r$ be a non-negative integer and $D$ be a simple digraph with vertex set $V$. Then $D$ is called an \textbf{$\alpha$-structure of rank $r$} if there exist $k \leq r$ arcs $e_1,\dots,e_k \in D$ such that for all $u,v \in V$, we have 
\[\text{$(u,v) \in D$ if and only if $(u,v) = (e_i^+, e_j^-)$ for some $1 \leq i \leq j \leq k$}.\]
The arcs $e_1, \dots, e_k$ are called \textbf{decisive arcs} of the $\alpha$-structure on $D$.
\end{definition}
The prime example of an $\alpha$-structure is a transitive tournament, where the edges of the longest directed path form the decisive arcs (in order). A different example is given on the right in \cref{fig:alpha-structures}, and we invite the reader to verify that both examples satisfy \cref{def:alpha}. 

\begin{figure}[H]
    \centering
    \scalebox{1.15}{\begin{tikzpicture}[]
\coordinate (v0) at (-6,0);
\coordinate (v1) at (-4,0);
\coordinate (v2) at (-2,0);
\coordinate (v3) at (0,0);
\coordinate (v4) at (2,0);
\coordinate (v5) at (4,0);
\coordinate (v6) at (6,0);
\fill (v0) circle (1pt);
\fill (v1) circle (1pt);
\fill (v2) circle (1pt);
\fill (v3) circle (1pt);
\fill (v4) circle (1pt);
\path (v0) edge[makerlastplayed] node[below] {$e_1$}(v1);
\path (v1) edge[makerlastplayed] node[below] {$e_2$} (v2);
\path (v2) edge[makerlastplayed] node[below] {$e_3$} (v3);
\path (v3) edge[makerlastplayed] node[below] {$e_4$} (v4);
\path (v0) edge[played] [bend left=20] (v2);
\path (v0) edge[played] [bend left=30] (v3);
\path (v1) edge[played] [bend left=20] (v3);
\path (v0) edge[played] [bend left=40] (v4);
\path (v1) edge[played] [bend left=30] (v4);
\path (v2) edge[played] [bend left=20] (v4);
\end{tikzpicture}}
    \hspace{1cm}
    \scalebox{1.15}{\begin{tikzpicture}[]
\foreach \x in {0,...,2} {
    \coordinate (v\x) at (2*\x,0);
    \coordinate (w\x) at (2*\x,2);
    \fill (v\x) circle (1pt);
    \fill (w\x) circle (1pt);
}
\path (0,2) edge[makerlastplayed] node[below] {$e_1$} (2,2);
\path (2,2) edge[makerlastplayed] node[below] {$e_2$} (4,2);
\path (0,0) edge[makerlastplayed] node[above] {$e_3$} (2,0);
\path (2,0) edge[makerlastplayed] node[above] {$e_4$} (4,0);
\path (w0) edge[played] (v2);
\path (w0) edge[played] (v1);
\path (w1) edge[played] (v1);
\path (w1) edge[played] (v2);
\path (w0) edge[played] [bend left=20]  (w2) ;
\path (v0) edge[played] [bend right=20] (v2);
\end{tikzpicture}
}
    \caption{Two $\alpha$-structures of rank 4 that are not of rank 3. The thick arcs represent the decisive arcs.}
    \label{fig:alpha-structures}
\end{figure}

We now summarise the intuition behind $\alpha$-structures discussed in \cite{CLEMENS201721}. 
Recall that OMaker's best known strategy involves building a long directed path \cite{BENELIEZER20121732}. Let $P = (e_1, \dots, e_k)$ be a directed path in $D$ with arcs $e_i = (v_i, v_{i+1})$ for $1 \leq i \leq k$. Suppose OMaker directs an edge $(v_{k+1}, w)$ for some $w \in V$. Then all pairs $\{w, v_i\}$ for $1 \leq i \leq k$ represent \textit{immediate threats} as the inclusion of any $(w, v_i)$ would close a directed cycle. Thus, OBreaker on their turn needs to ensure that all edges $(v_i ,w)$ are directed. Thus, in each round while OMaker can extend the directed path, the number of immediate threats OBreaker has to close immediately also increases by one.

\begin{figure}[H]
    \centering
    \scalebox{1.0}{
    \begin{tikzpicture}[]
\coordinate (v0) at (-6,0);
\coordinate (v1) at (-4,0);
\coordinate (v2) at (-2,0);
\coordinate (v3) at (0,0);
\coordinate (v4) at (2,0);
\coordinate (v5) at (4,0);
\coordinate (v6) at (6,0);
\fill (v0) circle (1pt) node[below] {$v_1$};
\fill (v1) circle (1pt) node[below] {$v_2$};
\fill (v2) circle (1pt) node[below] {$v_3$};
\fill (v3) circle (1pt) node[below] {$v_4$};
\fill (v4) circle (1pt) node[below] {$v_5$};
\fill (v5) circle (1pt) node[below] {$v_{k}$};
\fill (v6) circle (1pt) node[below] {$v_{k+1}$};
\path (v0) edge[makerlastplayed] (v1);
\path (v1) edge[makerlastplayed] (v2);
\path (v2) edge[makerlastplayed]  (v3);
\path (v3) edge[makerlastplayed] (v4);
\path (v4) edge[dotted] (v5);
\path (v5) edge[makerlastplayed] (v6);
\path (v0) edge[played] [bend left=20] (v2);
\path (v0) edge[played] [bend left=30] (v3);
\path (v1) edge[played] [bend left=20] (v3);
\path (v0) edge[played] [bend left=40] (v4);
\path (v1) edge[played] [bend left=30] (v4);
\path (v2) edge[played] [bend left=20] (v4);
\path (v0) edge[played] [bend left=50] (v5);
\path (v0) edge[played] [bend left=60] (v6);
\path (v1) edge[played] [bend left=40] (v5);
\path (v1) edge[played] [bend left=50] (v6);
\path (v2) edge[played] [bend left=30] (v5);
\path (v2) edge[played] [bend left=40] (v6);
\path (v3) edge[played] [bend left=20] (v5);
\path (v3) edge[played] [bend left=30] (v6);
\path (v4) edge[played] [bend left=20] (v6);
\end{tikzpicture}

    }
    \caption{The thick arcs represent the directed path created by OMaker whereas the thin arcs represent the immediate threats closed by OBreaker.}
    \label{fig:long-path}
\end{figure}

However, it is not sufficient for OBreaker to only close immediate threats. Suppose that OMaker builds two vertex-disjoint directed paths $P_1 = (v_1, \dots, v_{k+1})$ and $P_2 = (w_1, \dots, w_{\ell + 1})$ of lengths $k$ and $\ell$ respectively such that $k \ell > b$. If OMaker directs $(v_{k+1}, w_1)$, this creates $k\ell > b$ immediate threats which cannot all be closed by OBreaker on their next turn. The $\alpha$-structure prevents such a situation from happening. Furthermore, it demonstrates that building a long directed path is the \textit{best} strategy for OMaker in the sense that no matter how they play, in round $k$, OBreaker can direct at most $k$ edges to close immediate threats.

A useful property of $\alpha$-structures proved in \cite{CLEMENS201721} is that an induced subdigraph of an $\alpha$-structure is still an $\alpha$-structure.
\begin{proposition}[Lemma 2.3 in \cite{CLEMENS201721}]\label{prop:subgraph-of-alpha-is-alpha}
    Let $D$ be an $\alpha$-structure of rank $r$ on a vertex set $V$. Then for any subset $V' \subseteq V$, we have that $D[V']$ is an $\alpha$-structure of rank $r$.
\end{proposition}

As previously mentioned, one key property of $\alpha$-structures is that they cannot be made cyclic with the addition of a single arc. 
\begin{proposition}[Proposition 2.5 in \cite{CLEMENS201721}]\label{prop:alpha-structure-remains-acyclic}
    If $D$ is an $\alpha$-structure, then for every available $e \in \mathcal{A}(D)$ we have that $D + e$ is acyclic.
\end{proposition}
The following result provides an upper bound on the number of arcs required to form a new $\alpha$-structure when an additional arc is added to an existing $\alpha$-structure.
\begin{proposition}[Lemma 2.7 in \cite{CLEMENS201721}]\label{prop:addalpha}
Let $D$ be an $\alpha$-structure of rank $r$ on a vertex set $V$, and let $e \in \mathcal{A}(D)$ be an available arc. Then there exists a set $\{f_1,\dots,f_t\} \subseteq \mathcal{A}(D + e)$ of at most $\min\{r, |V| - 2\}$ available arcs such that $D' = D \cup \{ e, f_1, \dots, f_t\}$ is an $\alpha$-structure of rank $r+1$. Moreover, $D'^+ = D^+ \cup \{e^+\}$ and $D'^- = D^- \cup \{e^-\}$.
\end{proposition}

\begin{figure}[H]
    \centering
    \scalebox{0.99}{\begin{tikzpicture}[]\coordinate (v0) at (-6,0);
\coordinate (v1) at (-4,0);
\coordinate (v2) at (-2,0);
\coordinate (v3) at (0,0);
\coordinate (v4) at (2,0);
\coordinate (v5) at (4,0);
\fill (v0) circle (1pt);
\fill (v1) circle (1pt);
\fill (v2) circle (1pt);
\fill (v3) circle (1pt);
\fill (v4) circle (1pt);
\fill (v5) circle (1pt);
\path (v0) edge[makerlastplayed] node[below] {$e_1$}(v1);
\path (v1) edge[makerlastplayed] node[below] {$e_2$} (v2);
\path (v2) edge[makerlastplayed] node[below] {$e_3$} (v3);
\path (v3) edge[makerlastplayed] node[below] {$e_4$} (v4);
\path (v4) edge[makerlastplayed] node[below] {$e=e_5$} (v5);
\path (v0) edge[played] [bend left=20] (v2);
\path (v0) edge[played] [bend left=30] (v3);
\path (v1) edge[played] [bend left=20] (v3);
\path (v0) edge[played] [bend left=40] (v4);
\path (v1) edge[played] [bend left=30] (v4);
\path (v2) edge[played] [bend left=20] (v4);
\path (v0) edge[breakerlastplayed] [bend left=50] (v5);
\path (v1) edge[breakerlastplayed] [bend left=40] (v5);
\path (v2) edge[breakerlastplayed] [bend left=30] (v5);
\path (v3) edge[breakerlastplayed] [bend left=20] (v5);
\end{tikzpicture}}
    \hspace{1cm}
    \scalebox{0.99}{\begin{tikzpicture}[]
\foreach \x in {0,...,2} {
    \coordinate (v\x) at (2*\x,0);
    \coordinate (w\x) at (2*\x,2);
    \fill (v\x) circle (1pt);
    \fill (w\x) circle (1pt);
}
\coordinate (v3) at (6,0);
\fill (v3) circle (1pt);
\path (w0) edge[played] (v2);
\path (w0) edge[played] (v1);
\path (w1) edge[played] (v1);
\path (w1) edge[played] (v2);
\path (w0) edge[played] [bend left=20]  (w2) ;
\path (v0) edge[played] [bend right=20] (v2);
\path (w0) edge[breakerlastplayed] (v3);
\path (w1) edge[breakerlastplayed] (v3);
\path (v0) edge[breakerlastplayed] [bend right=40] (v3);
\path (v1) edge[breakerlastplayed] [bend right=20] (v3);
\path (0,2) edge[makerlastplayed] node[below] {$e_1$} (2,2);
\path (2,2) edge[makerlastplayed] node[below] {$e_2$} (4,2);
\path (0,0) edge[makerlastplayed] node[above] {$e_3$} (2,0);
\path (2,0) edge[makerlastplayed] node[above] {$e_4$} (4,0);
\path (4,0) edge[makerlastplayed] node[above] {$e=e_5$} (6,0);
\end{tikzpicture}}
    \caption{OBreaker directs at most 4 arcs (dashed) to incorporate OMaker's arc $e$ into an $\alpha$-structure of rank 4 to create an $\alpha$-structure of rank 5.}
    \label{fig:add-alpha-structures}
\end{figure}

Next, we formally define what a safe digraph is. Our definition of a safe digraph is original and seeks to capture the digraph invariants preserved by OBreaker in the strategy in~\cite{CLEMENS201721}.
\begin{definition}\label{defn:safe}
    For a digraph $D$ and $i,s,\delta \in \Z_{\geq 0}$, we say that $D$ is \textbf{$(i,s,\delta)$-safe} if there exist $A,B \subseteq V$ such that 
    \begin{enumerate}[label=\normalfont(\roman*)]
        \item $(A,B)$ is a UDB,
        \item $D[V\setminus B]$ is an $\alpha$-structure and $D[V\setminus B]^+ \subseteq A$,
        \item $D[V\setminus A]$ is an $\alpha$-structure and $D[V\setminus A]^- \subseteq B$,
        \item $\min(|A|, |B|) \geq s$,
        \item If $A \cup B \neq V$, then $2s + \delta + i = |A| + |B|$ and the $\alpha$-structures $D[V\setminus B]$ and $D[V\setminus A]$ have ranks $|A| - s $ and $|B| - s - \delta$ respectively.
    \end{enumerate}
We say that $D$ is \textbf{safe} if $D$ is $(i,s,\delta)$-safe for some $i,s,\delta \in \Z_{\geq 0}$.
\end{definition}
Another way of understanding \cref{defn:safe}~(v) is that if $k$ and $\ell$ denote the ranks of $D[V\setminus B]$ and $D[V\setminus A]$ respectively, then $|A| = s + k$ and $|B| = s + \ell + \delta$ so that $\delta = (|B| - \ell) - (|A| - k)$. As we seek to maximise $\min(|A| - k, |B| - \ell)$, we expect that OBreaker maintains $\delta$ to be relatively small so that the two quantities are  balanced. A depiction of a safe digraph is given in \cref{pic:safe}.

\begin{figure}[h]
\centering
\begin{tikzpicture}
    \draw (0.5,1.7) ellipse (3 and 1);
    \draw (0,-1.7) ellipse (2.5 and 1);
    \node at (-2,1.7) {\Large $A$};
    \node at (-2,-1.7) {\Large $B$};
    \node[fill, circle, inner sep=1.5pt] (A4) at (-1.4,1.7) {};
    \node[fill, circle, inner sep=1.5pt] (A3) at (-0.3,1.7) {};
    \node[fill, circle, inner sep=1.5pt] (A2) at (0.8,1.7) {};
    \node[fill, circle, inner sep=1.5pt] (A1) at (1.9,1.7) {};
    \node[fill, circle, inner sep=1.5pt] (A0) at (3, 1.7) {};
    \node[fill, circle, inner sep=1.5pt] (B1) at (-1.4,-1.7) {};
    \node[fill, circle, inner sep=1.5pt] (B2) at (-0.3,-1.7) {};
    \node[fill, circle, inner sep=1.5pt] (B3) at (0.8,-1.7) {};
    \node[fill, circle, inner sep=1.5pt] (B4) at (1.9,-1.7) {};
    \node[fill, circle, inner sep=1.5pt] (C1) at (3, 0) {};
    \node[fill, circle, inner sep=1.5pt] (C2) at (-3,0) {};
    \node[fill, circle, inner sep=1.5pt] (C4) at (-3,1.7) {};
    \node[fill, circle, inner sep=1.5pt] (C5) at (3, -1.7) {};
    \node[fill, circle, inner sep=1.5pt] (C6) at (-3,-1.7) {};    
    \node[fill, circle, inner sep=1.5pt] (v) at (-3,-0) {};
    \path (A4) edge[makerlastplayed]  (A3) ;
    \path (A3) edge[makerlastplayed]  (A2) ;
    \path (A4) edge[played, bend left = 20]  (A2) ;
    \path (B2) edge[makerlastplayed]  (B1) ;
    \path (v) edge[makerlastplayed] (A4) ;
    \path (v) edge[played] (A3) ;
    \path (v) edge[played] (A2) ;
    \foreach \x in {0,...,4} {
    \foreach \y in {1,...,4} {
        \path (A\x) edge[played]  (B\y) ;
    }
}
\end{tikzpicture}
\caption{A depiction of a $(4, 2, 1)$-safe digraph. The thick arcs represent the decisive arcs of the $\alpha$-structures $D[V\setminus B]$ and $D[V\setminus A]$ respectively. The $\alpha$-structure $D[V\setminus B]$ has rank 3 and the $\alpha$-structure $D[V\setminus A]$ has rank 1.}\label{pic:safe}
\end{figure}
Like $\alpha$-structures, safe digraphs cannot be made cyclic with the addition of a single arc.
\begin{proposition}\label{prop:safe-implies-no-threats}
    Let $D$ be a safe digraph. Then $D$ is acyclic and for all $e \in \mathcal{A}(D)$, we have that $D + e$ is acyclic.
\end{proposition}

\begin{proof}
Suppose $D$ is a safe digraph and let $(A, B)$ be a UDB which witnesses that $D$ is safe. Suppose on the contrary that for some available arc $e \in \mathcal{A}(D)$, we have that $D + e$ contains a directed cycle $C$. We have that $(A,B)$ is a UDB and for all $(v,w) \in D$, either $v \in A$ or $w \in B$. That is, all arcs between $A$ and $B$ go from $A$ to $B$, and there are no arcs from $B$ to $V\setminus B$ and no arcs from $V\setminus A$ to $A$. This implies that the edges of $C$ must lie entirely within $(D+e)[V\setminus A]$ or entirely within $(D+e)[V\setminus B]$. However, $D[V\setminus A]$ and $D[V\setminus B]$ are both $\alpha$-structures and so cannot be made cyclic with the addition of a single arc, by \cref{prop:alpha-structure-remains-acyclic}. The same argument, without adding $e$, shows that $D$ itself is acyclic in all cases.
\end{proof}    
The following proposition provides an upper bound on the number of arcs required to re-establish safety after an additional arc is added to a safe digraph. This is the main technical result underpinning our improved strategy for OBreaker. 
\begin{proposition}\label{prop:main-result-on-safety-transitions}
Let $i \in \Z_{>0}, s \in \Z_{\geq 0}$ and $\delta \in \{0,1\}$. Suppose that $D$ is an $(i-1, s, \delta)$-safe digraph witnessed by some UDB $(A,B)$.
Then the following hold for all $e \in \mathcal{A}(D)$.
\begin{enumerate}[label=\normalfont(\roman*)]
    \item 
    Let $p,q, x \in \Z_{\geq 0}$ be such that $p \leq \delta$, $q \leq 1 - \delta$, $2x+p + q < |V|-i-2s-\delta$, and assume that $A\cup B\neq V$. Then there exists $S \subseteq \mathcal{A}(D  + e)$ such that $|S| \leq x^2 + (p + q + 2s + \delta + i)x + (1 + p + q)(s+i) + p\delta$ and $(D + e) \cup S$ is $(i, s+x+p, \delta + q - p)$-safe, witnessed by some UDB $(\widetilde{A}, \widetilde{B})$ with $\widetilde{A} \cup \widetilde{B} \neq V$.
    \item There exists $S \subseteq \mathcal{A}(D + e)$ such that $(D +e) \cup S$ is $(i,s,\delta)$-safe and $|S| \leq |V|-s$.
\end{enumerate}
\end{proposition}
The number of parameters deserves a short explanation. 
The idea here is that in round $i$, if OMaker directs the edge $e$ then $S$ will be the set of edges which OBreaker directs. More specifically, OBreaker will direct the edges given by \cref{prop:main-result-on-safety-transitions}~(i) while they are building up the UDB which witnesses that the digraph is safe. Once the UDB is sufficiently large, OBreaker can easily maintain that the digraph is safe until the end of the game by repeatedly directing the edges given by \cref{prop:main-result-on-safety-transitions}~(ii). Thus, our focus is on \cref{prop:main-result-on-safety-transitions}~(i). 

Following the motivation at the beginning of Section~2, OBreaker wants to increase the parameter $s$ as much as possible in any given round. Observe that \cref{prop:main-result-on-safety-transitions}~(i) guarantees that the new digraph is $(i,s+x+p,\delta+q-p)$-safe so OBreaker should choose $x$ as large as possible in each round. Sometimes, OBreaker can direct enough edges to add an additional vertex to one of $A$ or $B$ but not both (so $x$ is already as large as possible). This case is handled by the auxiliary parameters $\delta, p, q \in \{0, 1\}$. Recalling \cref{defn:safe}~(v), by having $\delta \in \{0,1\}$, we allow $B$ to have, in a sense, up to one \textit{extra} vertex compared to $A$. When $\delta = 0$, OBreaker can add an extra vertex to $B$ by setting $q = 1$ and this results in $\delta = 1$. If $B$ already has an extra vertex (signified by $\delta = 1$), then OBreaker can add an additional vertex to $A$ by setting $p = 1$. Since both $A$ and $B$ now have an extra vertex, we can increment $s$ by one and reset $\delta$ back to 0. We emphasise here that in OBreaker's strategy, there will be a significant number of rounds where OBreaker can only choose $x=0$, but with $p=1$ or $q=1$. Doing this allows us to obtain a meaningful improvement on the threshold bias which is why we make the effort to define and track these auxiliary parameters despite the increase in complexity. Further details will be provided in Section~3. 

\begin{proof}[Proof of \cref{prop:main-result-on-safety-transitions}]
For convenience, we denote $D' := D  + e$. Let $(A,B)$ be a UDB which witnesses that $D$ is $(i-1, s, \delta)$-safe. We construct a set of arcs $F_1 \subseteq \mathcal{A}(D')$ such that $|F_1| \leq |V| - s$ and  $D' \cup F_1$ is $(i,s,\delta)$-safe. Within the context of \cref{prop:main-result-on-safety-transitions} (ii), we choose $S = F_1$. We also make use of $F_1$ in proving \cref{prop:main-result-on-safety-transitions} (i) where we construct another set of arcs $F_2$ and choose $S = F_1 \cup F_2$. The way in which we construct $F_1$ depends on whether or not $A \cup B = V$.
\begin{itemize}
    \item  We first consider the case where $A \cup B = V$.  Since $(A,B)$ is a UDB, all available arcs are contained in $\mathcal{A}(D[A])$ or $\mathcal{A}(D[B])$. First, suppose that $e \in \mathcal{A}(D[A])$. Then by \cref{prop:addalpha}, there exist available arcs $F_1 \subseteq \mathcal{A}(D'[A])$  such that $(D' \cup F_1)[A]$ is an $\alpha$-structure and $|F_1| \leq |A| - 2$. We check that $(A, B)$ is a UDB which witnesses that $D' \cup F_1$ is $(i,s,\delta)$-safe. It is clear that \cref{defn:safe} (i) and (iv) hold, since $D$ is $(i-1,s,\delta)$-safe and $(A,B)$ is the UDB witnessing this. \cref{defn:safe} (ii) is satisfied since $V\setminus B = A$ and $(D' \cup F_1)[A]$ is an $\alpha$-structure. \cref{defn:safe} (iii) holds as $D[V\setminus A]$ was originally an $\alpha$-structure (with $D[V\setminus A]^- \subseteq B$) and we have $F_1 \cup \{e\} \subseteq \mathcal{A}(D[A])$ so that $(D' \cup F_1)[V\setminus A] = D[V\setminus A]$. Lastly, \cref{defn:safe} (v) is not applicable as $A \cup B = V$. Hence, $D' \cup F_1$ is $(i,s,\delta)$-safe. Finally, since $D$ is $(i-1,s,\delta)$-safe then $|A| \leq |V| - s$ so that indeed $|F_1| \leq |V|-s$. The case where $e \in \mathcal{A}(D[B])$ can be handled similarly.
\item Suppose instead that $A \cup B \neq V$. Let $k$ and $\ell$ denote the ranks of the $\alpha$-structures $D[V\setminus B]$ and $D[V\setminus A]$ respectively. Since $D$ is $(i-1,s,\delta)$-safe then we have that $k = |A| - s$, $\ell = |B| - s - \delta$ and $2s + \delta + i - 1= |A| + |B|$, see (v) of \cref{defn:safe}.  
\begin{itemize}
    \item  Assume first that $e \in \mathcal{A}(D[V\setminus B])$. By \cref{prop:addalpha}, there exist at most $k$ available arcs $E_1 \subseteq \mathcal{A}(D'[V\setminus B])$ such that $(D' \cup E_1)[V\setminus B]$ is an $\alpha$-structure of rank $k+1$ with $(D' \cup E_1)[V\setminus B]^+ = D[V\setminus B]^+ \cup \{e^+\}$ and $(D' \cup E_1)[V\setminus B]^- = D[V\setminus B]^- \cup \{e^-\}$. 
\begin{figure}[h]
\centering
\begin{tikzpicture}
\draw (0,1.7) ellipse (2.5 and 1);
\draw (0,-1.7) ellipse (2.5 and 1);
\node at (-2,1.7) {\Large $A$};
\node at (-2,-1.7) {\Large $B$};
\node[fill, circle, inner sep=1.5pt] (A4) at (-1.4,1.7) {};
\node[fill, circle, inner sep=1.5pt] (A3) at (-0.3,1.7) {};
\node[fill, circle, inner sep=1.5pt] (A2) at (0.8,1.7) {};
\node[fill, circle, inner sep=1.5pt] (A1) at (1.9,1.7) {};
\node[fill, circle, inner sep=1.5pt] (B1) at (-1.4,-1.7) {};
\node[fill, circle, inner sep=1.5pt] (B2) at (-0.3,-1.7) {};
\node[fill, circle, inner sep=1.5pt] (B3) at (0.8,-1.7) {};
\node[fill, circle, inner sep=1.5pt] (B4) at (1.9,-1.7) {};
\node[fill, circle, inner sep=1.5pt, label=left:$v$] (v) at (-3,-0) {};
\path (A4) edge[available]  (A3);
\path (A3) edge[available]  (A2);
\path (A4) edge[available, bend left = 20]  (A2);
\path (B2) edge[available]  (B1);
\path (v) edge[makerlastplayed] node[above left] {$e$} (A4);
\path (v) edge[breakerlastplayed] (A3);
\path (v) edge[breakerlastplayed] (A2);
\foreach \x in {1,...,4} {
    \foreach \y in {1,...,4} {
        \path (A\x) edge[available]  (B\y) ;
    }
}
\foreach \y in {1,...,4} {
    \path (v) edge[breakerlastplayed2]  (B\y) ;
}
\end{tikzpicture}
\caption{A depiction of OBreaker's response to OMaker in the case where $e \in \mathcal{A}(D[V\setminus B])$ and $e^+ \in V\setminus (A\cup B)$. The dashed arcs correspond to the arcs in $E_1$ and incorporate $e$ into the $\alpha$-structure $D[V\setminus B]$. The dotted arcs correspond to the arcs in $E_2$ and ensure that $(A \cup \{v\}, B)$ is a UDB.  }\label{pic:safe-proof-1}
\end{figure} 
We define $v$ to be $e^+$ if $e^+ \in V\setminus (A \cup B)$ and otherwise an arbitrary vertex in $V \setminus (A \cup B)$. For every $w\in B$, the arc $(w,v)\not\in D$ by \cref{defn:safe}~(iii). This means that the arc $(v,w)$ is an element of $D$ or is available. Thus, there 
are available arcs $E_2 \subseteq \mathcal{A}(D' \cup E_1)$ of the form $(v, w)$ (for $w \in B$) such that $(A \cup \{v\}, B)$ is a UDB in $D' \cup E_1 \cup E_2$ and $|E_2| \leq |B|$. A depiction is provided in \cref{pic:safe-proof-1}. In this case, we choose $F_1 = E_1 \cup E_2$. 

We claim that $(A \cup \{v\}, B)$ is a UDB which witnesses that $D' \cup F_1$ is $(i,s,\delta)$-safe. It is clear that \cref{defn:safe} (i) and \cref{defn:safe} (iv) are satisfied since $D$ is $(i-1,s,\delta)$-safe and $(A,B)$ is the UDB witnessing this with $\min(|A|,|B|) \geq s$. We have $(D' \cup F_1)[V \setminus B] = (D' \cup E_1)[V \setminus B]$ which is an $\alpha$-structure and $(D' \cup F_1)[V\setminus B]^+ \subseteq A \cup \{v\}$ so that \cref{defn:safe} (ii) is satisfied. Given that $D[V\setminus A]$ is an $\alpha$-structure of rank $\ell$ with $D[V\setminus A]^- \subseteq B$, and that $(\{e\} \cup F_1)^+ \subseteq A \cup \{v\}$, we have by \cref{prop:subgraph-of-alpha-is-alpha} that $(D' \cup F_1)[V\setminus (A \cup \{v\})]$ is an $\alpha$-structure of rank $\ell$. Furthermore, $(D' \cup F_1)[V\setminus (A \cup \{v\})]^- \subseteq B$ so \cref{defn:safe} (iii) is satisfied. Towards (v), we have from $D$ being $(i-1,s,\delta)$-safe that $2s + \delta + i - 1= |A| + |B|$, which implies $2s + \delta + i = |A \cup \{v\}| + |B|$. The rank of $(D' \cup F_1)[V \setminus B]$  is now $k + 1 = |A| - s + 1 =  |A \cup \{v\}| - s$, and the rank of $(D' \cup F_1)[V \setminus (A \cup \{v\})]$  is still $\ell = |B| - s - \delta$ (note that an $\alpha$-structure of rank $r$ may have fewer than $r$ decisive arcs by \cref{def:alpha}). Thus, \cref{defn:safe} (v) is satisfied and $D' \cup F_1$ is indeed $(i,s,\delta)$-safe. Finally, we have 
\begin{equation}\label{eqn:upperbound-on-F1-case1}
|F_1| \leq |E_1| + |E_2| \leq k + |B| \leq |V|- (|A| - k) = |V|-s.
\end{equation}

\item Assume now that  $e \notin \mathcal{A}(D[V\setminus B])$. Note that since $(A,B)$ is a UDB, this implies that $e \in \mathcal{A}(D[V\setminus A])$. This case can be handled similarly to obtain
\begin{equation}\label{eqn:upperbound-on-F1-case2}
|F_1| \leq \ell + |A| \leq |V|- (|B| - \ell) = |V|-s - \delta. 
\end{equation}
\end{itemize}
\end{itemize}
This proves \cref{prop:main-result-on-safety-transitions} (ii) and we now focus on proving \cref{prop:main-result-on-safety-transitions}  (i). 

Let $p,q, x \in \Z_{\geq 0}$ be such that $p \leq \delta$, $q \leq 1 - \delta$, $2x+p + q < |V|-i-2s-\delta$, and assume now that $A\cup B\neq V$. We construct $F_1$ as above (in the case $A\cup B\neq V$), but derive an alternative upper bound on its size. Let $\mu = 1$ if  $e \in \mathcal{A}(D[V\setminus B])$ and otherwise $\mu = 0$. Then by \eqref{eqn:upperbound-on-F1-case1} and \eqref{eqn:upperbound-on-F1-case2} we have $|F_1| \leq \mu(k +|B|) + (1-\mu)(\ell + |A|)$. Recalling from \cref{defn:safe}~(v) that $k = |A| - s$, $\ell = |B| - s - \delta$ and $|A| + |B| = 2s + \delta + i - 1$, we have that
\begin{equation}\label{eqn:alt-upperbound-F1}
    |F_1| \leq s + i - 1 + \mu\delta
    \leq s + i.
\end{equation}

For convenience, we denote $D'' := D' \cup F_1$ and let $(A', B')$ denote the UDB obtained above which witnesses that $D''$ is $(i,s,\delta)$-safe. That is, either $(A', B') = (A \cup \{v\}, B)$ or $(A', B') = (A, B \cup \{v\})$ depending on whether $e \in \mathcal{A}(D[V \setminus B])$ or $e \in \mathcal{A}(D[V \setminus A])$. We also let $k'$ denote the rank of the $\alpha$-structure $D''[V \setminus B']$ and  $\ell'$ denote the rank of the $\alpha$-structure $D''[V \setminus A']$. Here, we have $k' = k + \mu$ and $\ell' = \ell + 1 - \mu$.

We now construct a set of arcs $F_2 \subseteq \mathcal{A}(D'')$ such that $D'' \cup F_2$ is $(i,s+x + p, \delta + q - p)$-safe. 
Observe that $|A'| + |B'| = |A| + |B| +1 = 2s +\delta + i $, by the $(i-1,s,\delta)$-safety of $D$. Thus, 
\begin{equation}\label{eqn:sum-less-than-V}
|A'| + |B'| + 2x + p + q = 2s + 2x + p + q + \delta + i   < |V|, 
\end{equation}
since by assumption, $2x+p + q < |V|-i-2s-\delta$. This means that there exist disjoint sets of vertices $U, W \subseteq V \setminus (A' \cup B')$ where $|U| = x + p$ and $|W| = x + q$. 
Every arc $(u,v)\in U\times (B'\cup W)$ is either in $D''$ or available, by \cref{defn:safe}~(iii) and since $D''$ is $(i,s,\delta)$-safe. Similarly, we see that every arc $(v,w)\in (A'\cup U) \times W$ is either in $D''$ or available. Let $F_2$ be the set of all those arcs $(v,w)\in (A'\cup U) \times W$ and $(u,v)\in U\times (B'\cup W)$ which are available. Then clearly $F_2^+ \subseteq A' \cup U$, $F_2^- \subseteq B' \cup W$, $(A'
\cup U, B' \cup W)$ is a UDB in $D'' \cup F_2$, and
\begin{equation}\label{eqn:initial-upper-bound-on-F2}
|F_2| \leq (x+p)|B'|+ (x+q)|A'| + (x+p)(x+q). 
\end{equation}

We now verify that $(A' \cup U, B' \cup W)$ is a UDB which witnesses that $D'' \cup F_2$ is $(i,s+x + p, \delta + q - p)$-safe. We are using several times that 
\begin{align}\label{aux:002}
    D'' \text{ is safe and the UDB } (A',B') \text{ witnesses it,}
\end{align}
and we recall for the convenience of the reader that $|A'|=|A|+\mu = k+s+\mu$, $|B'|=|B|+1-\mu = \ell+s+\delta+1-\mu$, where we use $\mu$ as the indicator variable of whether $e \in \mathcal{A}(D[V\setminus B])$. 
Firstly, we have already argued that $(A'\cup U,B'\cup W)$ is a UDB, so (i) is satisfied. Next, \eqref{aux:002} implies that $D''[V \setminus B']$ and $D''[V \setminus A']$ are $\alpha$-structures with $D''[V \setminus B']^+ \subseteq A'$ and $D''[V \setminus A']^- \subseteq B'$, respectively, see parts (ii) and (iii) of \cref{defn:safe}. We have by \cref{prop:subgraph-of-alpha-is-alpha} that $D''[V \setminus (B' \cup W)]$ and $D''[V \setminus (A' \cup U)]$ remain $\alpha$-structures. Since $F_2^+ \subseteq A' \cup U$ and $F_2^- \subseteq B' \cup W$, we have that $(D'' \cup F_2)[V \setminus (B' \cup W)] = D''[V \setminus (B' \cup W)]$ and $(D'' \cup F_2)[V \setminus (A' \cup U)] = D''[V \setminus (A' \cup U)]$ are indeed $\alpha$-structures with $(D'' \cup F_2)[V \setminus (B' \cup W)]^+ \subseteq A' \cup U$ and $(D'' \cup F_2)[V \setminus (A' \cup U)]^- \subseteq B' \cup W$, respectively. Thus, \cref{defn:safe}~(ii) and \cref{defn:safe}~(iii) are satisfied. Towards \cref{defn:safe}~(iv), we have  $|A' \cup U| = s + k + \mu + x + p$ and $|B' \cup W| = s + \ell +  \delta + 1 - \mu +  x + q$. Noting that $k, \ell, q, \mu, 1-\mu \geq 0$ and $\delta \geq p$, we have $\min(|A' \cup U|, |B' \cup W|) \geq s + x + p$. Hence, \cref{defn:safe} (iv) is satisfied. Finally, since $D''$ is $(i, s, \delta)$-safe, we have that $2s + \delta + i = |A'| + |B'|$. Thus, we have $|A' \cup U| + |B' \cup W| = |A'| +|B'|+ 2x + p + q = 2s + 2x + p + q + \delta + i = 2(s + x + p) + (\delta + q - p) + i$. Moreover, $(D'' \cup F_2)[V\setminus (B' \cup W)]$ is an $\alpha$-structure of rank $k' = |A'| - s =  |A' \cup U| - (s + x + p)$ and $(D'' \cup F_2)[V\setminus (A' \cup U)]$ is an $\alpha$-structure of rank $\ell' = |B'| - s - \delta = |B' \cup W| - (s + x + p) - (\delta + q - p)$. Hence, \cref{defn:safe} (v) is also satisfied. It follows that $D'' \cup F_2$ is $(i,s+x + p, \delta + q - p)$-safe. 

Finally, we show that 
\begin{equation}\label{eq:F1F2}
    |F_1\cup F_2| \leq x^2 + (p + q + 2s + \delta + i)x + (p + q+1)(s+i) + p\delta.
    \end{equation} 
For this, we use \eqref{eqn:initial-upper-bound-on-F2}, together with $|A'| =  k' + s$ and $|B'| = \ell' + s + \delta$,  to obtain
\begin{align*}
    |F_2| &\leq (x + p)(\ell' + s + \delta) + (x + q)(k' + s) + (x + p)(x+q)\\
     &= x^2 + (p + q + 2s + \delta  + k' + \ell')x + p\ell' + qk' + s(p + q)  + p\delta + pq.
\end{align*}
Using the fact that $k' + \ell' = i$,  $p\ell' + qk' \leq (p+q)(k' + \ell')$, and $pq = 0$, we see that the above implies 
\begin{equation*}
    |F_2| \leq x^2 + (p + q + 2s + \delta  + i)x + (p + q)(s + i)  + p\delta.
\end{equation*}
Combining this with \eqref{eqn:alt-upperbound-F1} gives~\eqref{eq:F1F2}. 

To see that Part (i) holds, choose $S = F_1 \cup F_2$. Then~\eqref{eq:F1F2} gives the upper bound on $|S|$, and $(D+e)\cup S = D''\cup F_2$ is $(i,s+x + p, \delta + q - p)$-safe, witnessed by the UDB $(A' \cup U, B' \cup W)$. Here, note that~\eqref{eqn:sum-less-than-V} implies that $(A' \cup U) \cup (B' \cup W) \neq V$.
\end{proof}

\section{An Improved Strategy for OBreaker in the  Oriented-cycle Game}
In our improved strategy, OBreaker repeatedly directs the edges given by \cref{prop:main-result-on-safety-transitions} (i) until $s \geq n - b$. Following this, they maintain that the digraph is safe by repeatedly applying \cref{prop:main-result-on-safety-transitions} (ii) until the end of the game. Thus, to fully specify our strategy, we need to provide the values of $x, p, q$ used when we apply \cref{prop:main-result-on-safety-transitions} (i).

Let $x_i, p_i, q_i$ denote the values of $x, p, q$ chosen by OBreaker in round $i$ within the context of \cref{prop:main-result-on-safety-transitions} (i). 
Intuitively, OBreaker should choose $x_i$ to be as large as possible as $s$ will directly increase by $x_i$ (and $p\le 1$). When OBreaker cannot direct sufficiently many edges to add another vertex to both $A$ and $B$ but adds an additional vertex to only $A$ (respectively $B$), the quantity $p_i$ (respectively $q_i$) is equal to 1. The quantities $x_i, p_i, q_i$ need to be chosen such that OBreaker directs at most $b$ edges. We define several functions and quantities which we use in specifying the values of $x_i, p_i, q_i$ chosen by OBreaker in our improved strategy.
\begin{definition}\label{defn:round-cost}
We define the function $\costf: \R^5_{\geq 0}  
\to \R_{\geq 0}$ by 
\begin{align*}
Q(x,\diffContribution, s,\delta,i) := x^2 + (\diffContribution + 2s + \delta + i)x + 
(1+\diffContribution)(s+i) + \diffContribution\delta.
\end{align*}
\end{definition}
If OBreaker directs the arcs in $S$ of \cref{prop:main-result-on-safety-transitions}~(i) in round $i$, then they direct at most $\costf(x_i, p_i + q_i, s, \delta, i)$ edges. 
Thus, to ensure that OBreaker directs at most $b$ edges, we require 
\begin{equation}\label{eqn:bound-for-x-alpha-beta-using-costf}
   \costf(x_i, p_i + q_i, s, \delta, i) \leq b.
\end{equation}
Notice that in the definition of $Q$, we do not restrict ourselves to integer arguments. However, $Q$ is indeed integer-valued when its arguments are integers, and we will only apply \cref{prop:main-result-on-safety-transitions}~(i) with integer values. As an initial goal, we seek the largest $x_i \in \R_{\geq 0}$ satisfying \eqref{eqn:bound-for-x-alpha-beta-using-costf} for $p_i=q_i=\delta=0$. 
\begin{definition}\label{defn:G}
For a positive integer $b$, we define the functions $\Gb: [0, 4b] \to \R$ and $\gb: [1, 4b] \to \R$ by 
$\Gb(t) := -\frac{t}{2} + \sqrt{bt - \frac{t^2}{4}}$ and $\gb(t) := \Gb(t) - \Gb(t-1)$.
\end{definition}

In our improved strategy, OBreaker chooses $x_i = \floor{\gb(i)}$. 
\begin{proposition}\label{prop:cost-g-at-most-b}
For positive integers $b$ and  $1 \leq i \leq 2b + 1/2$, we have $
\costf(\gb(i), 0, \Gb(i-1), 0, i) = b + \frac{1}{2}.$
\end{proposition}
We note that this does not quite satisfy \eqref{eqn:bound-for-x-alpha-beta-using-costf} but will be sufficient in later arguments.

\begin{proof}
Combining the definition of $\gb(i)$ with \cref{defn:round-cost}, and noticing cancellations, we obtain 
\begin{align*}
    \costf(\gb(i), 0, \Gb(i-1), 0, i) 
        &= \gb(i)^2 + (2\Gb(i-1) + i)\gb(i) + \Gb(i-1) + i\\
        &= \frac{(2\Gb(i) + i)^2 - (2\Gb(i-1) + i-1)^2 + 1 + 2i}{4}\\ 
        &= \frac{4bi - i^2 - 4b(i-1) + (i-1)^2 + 1 + 2i}{4} \\ &= b+\frac{1}{2},
\end{align*}
where in the third equality we use $(2\Gb(i) + i)^2 = 4bi - i^2$ which follows from \cref{defn:G}.
\end{proof}
It turns out that $g_b$ is invertible over a restricted domain where it also has an explicit formula for its inverse. This will be critical later when we try to quantify the values of the bias for which our new strategy is winning. 
\begin{proposition}\label{prop:g-is-invertible}
Let $b$ be a positive integer. The function $\gb$ is invertible on the interval $[1, 2b + 1/2]$ where its inverse $\gbinv: [\gb(2b+1/2), \gb(1)] \to \left[1, 2b + 1/2
\right]$ satisfies
\begin{align*}
    \gbinv(x) = 2b\left(1 + \frac{1}{4b} - \sqrt{\frac{(2x+1)^2}{(2x+1)^2+1}}\sqrt{1-\frac{(2x+1)^2+1}{16b^2}}\right).
\end{align*}
\end{proposition}
We remark here, for the reader's convenience, that $\gb(2b+1/2)< 0$, and  that we will only consider $0 \leq x \leq \gb(1)$ which corresponds to $1 \leq \gbinv(x) \leq 2b + \frac{1-\sqrt{8b^2-1}}{2} < 2b + 1/2$. 
We also note that $\gb(1) = \Gb(1)\approx \sqrt{b}$, so for $0 \leq x \leq \gb(1)$ and $b$ large, we observe that \[\gbinv(x) \approx 2b\left(1 - \sqrt{\frac{(2x+1)^2}{(2x+1)^2+1}}\right).\]
\begin{proof}
It is clear that $\gb$ is continuous on $\left[1,2b + 1/2\right]$. We show that $\gb$ is strictly decreasing in the interior of this interval. For $t \in \left(1, 2b + 1/2\right)$, $\gb$ is differentiable and we have 
\begin{align*}
        \gb'(t) &= \Gb'(t) - \Gb'(t-1) 
        = \frac{2b-t}{2\sqrt{4bt - t^2}} - \frac{2b-t+1}{2\sqrt{4b(t-1) - (t-1)^2}} 
        \leq \frac{-1}{2\sqrt{4bt - t^2}}
        < 0.
\end{align*}    
Hence, we deduce that it is invertible on the interval $\left[1, 2b +1/2\right]$. We denote $x_b := \gb(t)$. To obtain an expression for $\gb^{-1}$, we derive an expression for $t$ in terms of $x_b$. Recalling that $g_b(t) = \Gb(t) - \Gb(t-1)$, we have 
\begin{equation*}
    x_b = -\frac{1}{2} + \sqrt{bt - \frac{t^2}{4}} - \sqrt{b(t-1) - \frac{(t-1)^2}{4}},
\end{equation*}
which is equivalent to
\begin{equation*}
    2x_b + 1 =  \sqrt{4bt - t^2} - \sqrt{4b(t-1) - (t-1)^2}.
\end{equation*}
Squaring and rearranging, we obtain that 
\begin{equation*}
    t^2 - (4b+1)t + \frac{1}{2}((2x_b + 1)^2 + 1 + 4b) = - \sqrt{4bt - t^2}\sqrt{4b(t-1) - (t-1)^2}.
\end{equation*}
Squaring again, it follows that 
\begin{align*}
&~\phantom{=}t^4  - (8b + 2)t^3 + (16b^2 + 8b + 1 + ((2x_b + 1)^2 + 1 + 4b))t^2 \\
     &\phantom{=}- (4b+1)((2x_b + 1)^2 + 1 + 4b)t  + \frac{1}{4}((2x_b + 1)^2 + 1 + 4b)^2 \\
          &=  t^4        - (8b + 2)t^3                + (16b^2 + 12b + 1)t^2  - (16b^2 + 4b)t,
\end{align*}
which in turn is equivalent to
\begin{equation}\label{eqn:quad-in-t-p}
    t^2 - (4b+1)t + \frac{((2x_b + 1)^2 + 1 + 4b)^2}{4((2x_b + 1)^2+1)} = 0.
\end{equation}
The unique root of~\eqref{eqn:quad-in-t-p} in $\left[1,2b + 1/2\right]$ is 
\begin{align*}
    \gb^{-1}(x) &= 2b+\frac{1}{2} - \frac{1}{2}\, \sqrt{(4b+1)^2 - \frac{((2x + 1)^2 + 1 + 4b)^2}{((2x + 1)^2+1)}} \\ 
    &= 2b\left(1 + \frac{1}{4b} - \sqrt{\frac{(2x+1)^2}{(2x+1)^2+1}}\sqrt{1 - \frac{(2x+1)^2+1}{16b^2}}\right). \tag*{\qedhere} 
\end{align*}
\end{proof}
We define three integer quantities $\taub{i}, \pb{i}, \smb{i}$ which we use in describing our strategy for OBreaker. We briefly explain the intuition behind each quantity after it is defined.
\begin{definition}\label{defn:tau_i}
For integers $b \geq 3$ and $1 \leq i \leq \gDomUB$, we define
    \begin{align*}
        \taub{i} &:= \begin{cases}
        1, \text{ if $\gb(i) - \floor{\gb(i)} \geq \frac{1}{2} + \left(\frac{16b}{i} - 4\right)^{-1/2}$}, \\ 
        0, \text{ otherwise.}
    \end{cases}
    \end{align*}
\end{definition}
In our strategy for OBreaker, $\taub{i} = 1$ corresponds to when OBreaker chooses $p_i$ or $q_i$ to be 1, alternating between the two cases to ensure that $\delta \in \{0,1\}$, where we recall that $\delta = (|B| - \ell) - (|A| - k)$ and that we seek to maximise $s = \min(|A| - k, |B| - \ell)$. 

Recalling that the range of $\gbinv$ is contained in $[1, 2b + 1/2]$, we have $1 \leq i \leq \gDomUB \leq 2b$  so that ${1}/{2} + \left({16b}/{i} - 4\right)^{-1/2} \in [1/2, 1]$. Essentially, if the fractional part of $\gb(i)$ is sufficiently large, then OBreaker can choose $x_i = \floor{\gb(i)}$ with either $p_i = 1$ or $q_i = 1$ in our strategy. This is stated more precisely in \cref{prop:cost-g-floor-at-most-b}. The requirement that the fractional part is at least $1/2$ can loosely be understood in the following way: if $A$ and $B$ have almost the same size, then to add another vertex to one of $A$ or $B$, we must have been able to add at least \textit{half} a vertex to both $A$ and $B$. 
\begin{definition}\label{defn:p_i-and-s_i}
    For integers $b \geq 3$ and $0 \leq i  \leq \gDomUB$, we define $\pb{i} := \sum_{j=1}^{i} \taub{j}\pmod{2}$ and $\smb{i} := \sum_{j=1}^{i} \floor{\gb(j)} +  \sum_{j=1}^{i} \taub{j} \pb{j-1}$.
\end{definition}
We show later that at the end of round $i$ in our strategy for OBreaker, the digraph will be $(i, \smb{\myfloor{i}}, \pb{\myfloor{i}})$-safe where $\myfloor{i} := \min (i, \gDomUB)$. 

Suppose that the digraph $D$ is $(i-1, \smb{i-1}, \pb{i-1})$-safe at the start of round $i$. The following proposition provides an upper bound on the number of edges OBreaker needs to then direct to apply \cref{prop:main-result-on-safety-transitions} (i) with $x_i = \floor{\gb(i)}$, $p_i = \pb{i-1}\taub{i}$, and $q_i = (1-\pb{i-1})\taub{i}$.
\begin{proposition}\label{prop:cost-g-floor-at-most-b}
For integers $b \geq 3$ and $1 \leq i \leq \gDomUB$, we have 
\begin{align*}
    \costf(\floor{\gb(i)}, \taub{i}, \smb{i-1}, \pb{i-1}, i) \leq b. 
\end{align*}
\end{proposition}
\begin{proof}
We show that 
\begin{align}\label{eqn:eqn-to-show-1}
\smb{i-1} \leq \Gb(i-1)- \pb{i-1}/2
\end{align}
and 
\begin{align}\label{eqn:eqn-to-show-2}
    \costf(\floor{\gb(i)}, \taub{i}, \Gb(i-1)- \pb{i-1}/2, \pb{i-1}, i)
    &\leq \costf(\gb(i), 0, \Gb(i-1), 0, i).
\end{align}
This is sufficient as combining the monotonicity of $\costf$ in $s$, \eqref{eqn:eqn-to-show-1} and \eqref{eqn:eqn-to-show-2} yields
\begin{align*}
    \costf(\floor{\gb(i)}, \taub{i}, \smb{i-1}, \pb{i-1}, i) 
    &\leq \costf(\floor{\gb(i)}, \taub{i}, \Gb(i-1)- \pb{i-1}/2, \pb{i-1}, i) \\
    &\leq \costf(\gb(i), 0, \Gb(i-1), 0, i)\\
    &= b+1/2,
\end{align*}
by \cref{prop:cost-g-at-most-b}. Since $\costf$ is integer-valued when its inputs are integers, it then follows that
$$\costf(\floor{\gb(i)}, \taub{i}, \smb{i-1}, \pb{i-1}, i) \leq b.$$
We prove~\eqref{eqn:eqn-to-show-1} 
by induction on $i$.  The case where $i=1$ holds trivially. Assume now that $\smb{i-1} \leq \Gb(i-1) - \pb{i-1}/2$ for some arbitrary integer $1 \leq i < \gDomUB$. Then we have that
\begin{align*}
    \smb{i} &= \smb{i-1} + \floor{\gb(i)} + \taub{i} \pb{i-1}\\ 
            &\leq \Gb(i-1) - \pb{i-1}/2 + \floor{\gb(i)} + \taub{i} \pb{i-1}\\  
            &\leq \Gb(i) - (\pb{i-1}/2 + \gb(i) - \floor{\gb(i)} - \taub{i} \pb{i-1}),
\end{align*}
by \cref{defn:p_i-and-s_i}, the inductive hypothesis, and \cref{defn:G}, respectively. Recall that 
$$\pb{i} = \sum_{j=1}^{i} \taub{j}\pmod{2} = \pb{i-1} + \taub{i}\pmod{2}$$ 
and that  $\taub{i} = 1$ implies $\gb(i) - \floor{\gb(i)} \geq {1}/{2}$. It can then be  verified that for all $\pb{i-1}, \taub{i} \in \{0,1\}$, we have
\begin{align*}
\pb{i}/2 = \frac{1}{2} (\pb{i-1} + \taub{i} \text{ (mod 2)})\leq \pb{i-1}/2 + \gb(i) - \floor{\gb(i)} - \taub{i}\pb{i-1}.    
\end{align*}
Thus, $\smb{i} \leq \Gb(i) - \pb{i}/2$ as desired.

Next, we prove \eqref{eqn:eqn-to-show-2} by showing that $R-L \geq 0$ where $R:=\costf(\gb(i), 0, \Gb(i-1), 0, i)$ and $L:=\costf(\floor{\gb(i)}, \taub{i}, \Gb(i-1)- \pb{i-1}/2, \pb{i-1}, i)$. \cref{defn:round-cost}, with some rearranging, yields
\begin{align*}
R &= \gb(i)^2 + (2\Gb(i-1) + i)\gb(i) + \Gb(i-1) + i \\
L &= \floor{\gb(i)}^2 + (\taub{i} + 2\Gb(i-1) + i)\floor{\gb(i)} + (1 + \taub{i})(\Gb(i-1)- \pb{i-1}/2 + i) + \taub{i}\pb{i-1}.
\end{align*}
Writing $\epb{i}$ for the fractional part $\gb(i) - \floor{\gb(i)}$ of $\gb(i)$, the subtraction can be written as 
\begin{align*}
R-L &=  \epb{i}^2  + \epb{i}(2\floor{\gb(i)} + 2\Gb(i-1) + i) - \taub{i}(\floor{\gb(i)} + \Gb(i-1) + i + \pb{i-1}/2) + \pb{i-1}/2.
\end{align*}
Hence, to prove that $R-L\ge 0$, it suffices to show that
\begin{equation}\label{eqn:eqn-to-show-3}
       \epb{i}(2\floor{\gb(i)} + 2\Gb(i-1) + i) - \taub{i}(\floor{\gb(i)} + \Gb(i-1) + i) \geq (\taub{i}-1)\pb{i-1}/2 - \epb{i}^2.
\end{equation}
If $\epb{i} < {1}/{2} + \left({16b}/{i} - 4\right)^{-1/2}$, we have $\taub{i} = 0$ by \cref{defn:tau_i}, and~\eqref{eqn:eqn-to-show-3} becomes
$$
\epb{i}(2\floor{\gb(i)} + 2\Gb(i-1) + i)  \geq -\pb{i-1}/2 - \epb{i}^2.
$$
This clearly holds as $2\Gb(i-1) + i \geq 0$ by \cref{defn:G}. Otherwise, suppose that $\epb{i} \geq {1}/{2} + \left({16b}/{i} - 4\right)^{-1/2}$ so that $\taub{i} = 1$ and the right-hand side of~\eqref{eqn:eqn-to-show-3} is just $-\epb{i}^2$. Recalling that by \cref{defn:G}, we have $\Gb(i) = -{i}/{2} + \sqrt{bi - {i^2}/{4}}$, it is easily verified that
${i}{(4\Gb(i) + 2i)^{-1}} = \left({16b}/{i} - 4\right)^{-1/2}$ so that $\epb{i} \geq {1}/{2} + {i}{(4\Gb(i) + 2i)^ {-1}}$. 
Thus, $\epb{i}(2\Gb(i) + i) \geq \Gb(i) + i$
so that 
$\epb{i}(2\Gb(i) + i) -\Gb(i) - i \geq 0$. Recalling that $\Gb(i) = \gb(i) + \Gb(i-1)$, by \cref{defn:G}, this is equivalent to
\begin{equation}\label{eqn:prop3.7-4}
    \epb{i}(2\floor{\gb(i)} + 2\Gb(i-1) + i) - (\floor{\gb(i)}+ \Gb(i-1) + i) \geq \epb{i} - 2\epb{i}^2.
\end{equation}
Comparing the right-hand side of \eqref{eqn:eqn-to-show-3} and \eqref{eqn:prop3.7-4}, we see that to prove \eqref{eqn:eqn-to-show-3}, it suffices to show that
\begin{equation*}
    \epb{i} - \epb{i}^2 \geq 0.
\end{equation*}
However, this is clear as $0 \leq \epb{i} < 1$. Hence, we have proved \eqref{eqn:eqn-to-show-3} which in turn proves \eqref{eqn:eqn-to-show-2} and thus, the proposition.
\end{proof}

The following proposition essentially reformulates \cref{prop:main-result-on-safety-transitions} within the context of our improved strategy for OBreaker in the Oriented-cycle game. Let $M = M(b) := \floor{\gbinv(1/2)}$ which is the last round $i$ where $\gb(i) \ge 1/2$, and which is approximately $(2-{4}/{\sqrt{5}})b \approx 0.2111b$. Round $M$ is, roughly speaking, the last round in which OBreaker uses \cref{prop:main-result-on-safety-transitions}~(i); from round $M+1$ onward, OBreaker uses  \cref{prop:main-result-on-safety-transitions}~(ii). For $i \in \Z_{>0}$, we write $\myfloor{i} := \min (i, M)$. 
\renewcommand{\gDomUB}{\ensuremath{M}}

\begin{proposition}
\label{prop:specialised-result-on-safety-transitions}
Let positive integers $n,b$ be large enough with $b \leq n \leq \smb{\gDomUB} + b$. For some positive integer $i$, let $D$ be an $(i-1, \smb{\myfloor{i-1}}, \pb{\myfloor{i-1}})$-safe digraph on $n$ vertices, witnessed by the UDB $(A,B)$, and suppose additionally  that $A \cup B \neq V$ when $i \le \gDomUB$. Assume that $|\mathcal{A}(D)| > 0$ and let $e \in \mathcal{A}(D)$. Then either $|\mathcal{A}(D+e)| = 0$, or there exists $S \subseteq \mathcal{A}(D+e)$ of size at most $b$ such that $(D+e) \cup S$ is $(i, \smb{\myfloor{i}}, \pb{\myfloor{i}})$-safe, where, if $i \leq \gDomUB$, the UDB $(\widetilde{A}, \widetilde{B})$ witnessing the $(i, \smb{\myfloor{i}}, \pb{\myfloor{i}})$-safety satisfies 
$\widetilde{A} \cup \widetilde{B} \neq V$. 
\end{proposition}

\begin{proof}
Let $D$ be a digraph which is $(i-1, \smb{\myfloor{i-1}}, \pb{\myfloor{i-1}})$-safe, witnessed by some UDB $(A,B)$, and assume that $|\mathcal{A}(D  + e)| > 0$. We first consider the case where $i \leq \gDomUB$. In this case, we have $\myfloor{i-1}=i-1$. We seek to apply \cref{prop:main-result-on-safety-transitions}~(i) with $p_i = \pb{i-1}\taub{i}$, $q_i= (1-\pb{i-1})\taub{i}$ and $x_i = \floor{\gb(i)}$. Note that our choice of $p_i$ and $q_i$ ensures that $\pb{i} = \pb{i-1} + q_i - p_i$ for all $\pb{i-1}$ and $\taub{i}$. By assumption, we have $A \cup B \neq V$ so it just remains to check that $2x_i + p_i + q_i < |V| - i - 2s_{b,i-1} - \pb{i-1}$. First, it can be easily verified that 
\begin{equation}\label{eqn:i-M-upper-bound}
i \leq M = \floor{\gbinv(1/2)} \leq \left(2 - \frac{4}{\sqrt{5}}\right)b + O(1).
\end{equation}
Next, recall that in the proof of \cref{prop:cost-g-floor-at-most-b}, we showed that $s_{b,i-1} \leq \Gb(i-1)$. Since $\Gb$ is monotonically increasing for $1 \leq i \leq M$, we have
\begin{equation}\label{eqn:sb-M-upper-bound}
s_{b,i-1} \leq \Gb(M) \leq \left(\frac{3}{\sqrt{5}} - 1\right)b + O(1).
\end{equation}
Finally, since $\gb$ is monotonically decreasing for $1 \leq i \leq M$, we have 
\begin{equation}\label{eqn:x-1-upper-bound}
x_i \leq \gb(i) \leq \gb(1) \leq \sqrt{b}.
\end{equation}
Thus, by \eqref{eqn:i-M-upper-bound}, \eqref{eqn:sb-M-upper-bound} and \eqref{eqn:x-1-upper-bound}, together with $b \leq n$ and $\taub{i}, \pb{i-1} \in \{0,1\}$, we have that
$$ |V| - i - 2s_{b,i-1} - \pb{i-1} \geq \left(1 - \frac{2}{\sqrt{5}}\right)b - O(1) > 0.1b > 2\sqrt{b} + 1 \geq 2x_i + p_i + q_i$$
for all sufficiently large $b$. Hence, by \cref{prop:main-result-on-safety-transitions}~(i), there exists $S \subseteq \mathcal{A}(D + e)$ such that  $|S|\le \costf(\floor{\gb(i)}, \taub{i}, \smb{i-1}, \pb{i-1}, i)$ and  $(D + e) \cup S$ is $(i, \smb{i-1} + x_i +p_i, \pb{i-1}+q_i-p_i)$-safe, witnessed by some UDB $(\widetilde{A}, \widetilde{B})$ with $\widetilde{A} \cup \widetilde{B} \neq V$. By \cref{prop:cost-g-floor-at-most-b}  we therefore have that $|S| \leq b$. By definition of $\smb{{i}}$ and $\pb{{i}}$, see \cref{defn:p_i-and-s_i}, and since $\myfloor{i}=i$, the claim follows in this case. 

Next, we consider the case where $i > \gDomUB$. In this case, we have that $\myfloor{i} = \myfloor{i-1} = \gDomUB$. By \cref{prop:main-result-on-safety-transitions}~(ii) and the assumption that $\smb{\gDomUB} \geq n-b$, there exists $S \subseteq \mathcal{A}(D  + e)$ such that $(D  +e) \cup S$ is $(i,\smb{\myfloor{i}},\pb{\myfloor{i}})$-safe and $|S| \leq n-\smb{\myfloor{i-1}} = n - \smb{\gDomUB} \leq b$.
\end{proof}
For technical reasons, we will apply \cref{prop:specialised-result-on-safety-transitions} with $b' = b - 1$ later on when formally describing our improved strategy for OBreaker. Furthermore, the index $i$ will track the number of times we have applied \cref{prop:specialised-result-on-safety-transitions}, rather than the actual round number. The reason for this distinction is that we may sometimes apply \cref{prop:specialised-result-on-safety-transitions} twice in a given round to ensure that OBreaker directs at least one edge.

In order to apply \cref{prop:specialised-result-on-safety-transitions}, we require $n \leq \smb{\gDomUB} + b$. Noting that we only consider $b = \Theta(n)$ as it is known that $t(n,\mathcal{C}) = \Theta(n)$, we derive a linear lower bound on $\smb{\gDomUB}$.
\begin{proposition}\label{prop:ineq-for-CB}
For every integer $b \geq 3$, we have that $\smb{\gDomUB} \geq \Cb b$ where 
\begin{align*}
    \Cb := - \frac{2+g_b(1)}{2b} + \sum_{a=0}^{\floor{g_b(1)}-1}  \left(2 - \sqrt{\frac{(2(a+ \phi_{b,a})+1)^2}{(2(a+ \phi_{b,a})+1)^2+1}} - \sqrt{\frac{(2a+3)^2}{(2a+3)^2+1}}\right)
\end{align*}
and 
\begin{align*} 
    \phi_{b,a} &:= 
        \frac{1}{2} +  \left(\frac{16b}{g^{-1}_{b}(a + 1/2)} - 4\right)^{-1/2}  \text{ for } 0 \leq a \leq \floor{g_b(1)} - 1.
\end{align*}
\end{proposition}
As we seek as tight a lower bound as possible on $\smb{\gDomUB}$, this means that the definition for our coefficient $C_b$ is unfortunately quite technical. In practice, $C_b \approx 0.275$ for large values of $b$. 
We recall  that 
\begin{equation}\label{eqn:smb-gdomub}
\smb{\gDomUB} = \sum_{i=1}^{\gDomUB} \floor{\gb(i)} + \sum_{i=1}^{\gDomUB} \taub{i} \pb{i-1}, 
\end{equation}
by \cref{defn:p_i-and-s_i}. The second sum accounts for the fact that in some rounds we add an additional vertex to $A$ or $B$, but not both. We mention here that this contribution turns out to be a significant proportion of $\smb{\gDomUB}$, and thus has a direct effect on improving the threshold bias by $\Theta(n)$. 

Note that $1/2\le  \phi_{b,a}  < 1$ for $0\le a \le \floor{\gb(1)}-1$. The lower bound $1/2\le  \phi_{b,a}$ is clear. Regarding the upper bound, $\gbinv$ is decreasing so $\gbinv(a + 1/2) \leq \gbinv(1/2) < 2b$. Hence, $\phi_{b,a}$ is decreasing in $a$ with $\phi_{b,a} \leq  \phi_{b,0} <  1$.
\begin{proof}
We first derive alternative expressions for both sums on the right-hand side of \eqref{eqn:smb-gdomub} individually. We split the sum $\Sigma_1:= \sum_{i=1}^{\gDomUB} \floor{\gb(i)}$ according to the different values of $\floor{\gb(i)}$. 
The function $\floor{\gb(i)}$ is decreasing in $i$, since $\gb(i)$ is, with values between $\floor{\gb(M)} \geq 0$ and $\floor{\gb(1)}$. 
For $0\le a \le \floor{\gb(1)},$ let $I_a$ be the set of all integers $i\in [1,M]$ such that $\floor{\gb(i)} = a$. 
Then 
\begin{align*}
\Sigma_1  &= \sum_{a=1}^{\floor{\gb(1)}} a\cdot |I_a|  
\     =\ \floor{\gb(1)} \cdot \floor{\gbinv(\floor{\gb(1)})} +  \sum_{a=1}^{\floor{\gb(1)}-1}a \cdot \big(\floor{\gbinv(a)} - \floor{\gbinv(a+1)}\big). 
\end{align*}
Observe that this is a telescoping sum which simplifies to 
\begin{equation}\label{eqn:exact-alt-for-X}
    \Sigma_1 = \sum_{a=1}^{\floor{\gb(1)}} \floor{\gbinv(a)}.
\end{equation}
We now turn to $\Sigma_2:= \sum_{i=1}^{\gDomUB} \taub{i} \pb{i-1}$ which is the second sum in~\eqref{eqn:smb-gdomub}. 
Recalling \cref{defn:tau_i}, we have that $\pb{i-1}  = \sum_{j=1}^{i-1} \taub{j} \pmod{2}$. Let $1 \leq i_1 < i_2 < \dots < i_k  < \dots < i_m \leq \gDomUB$ be the values of $i$ for which  $\taub{i} = 1$. For all other values of $i$, we have $\taub{i} = 0$. Thus, we have 
\[\pb{i_k-1} = \sum_{j=1}^{i_k-1} \taub{j} \pmod{2} = \sum_{1 \leq j < k} \taub{i_j} \pmod{2} = k-1 \pmod{2}.\] 
Hence, we have that $\taub{i} \pb{i-1} = 1$ if and only if $i = i_k$ for some even $1 \leq k \leq m$. 
Therefore,
\begin{align*}
\Sigma_2 =  \sum_{i=1}^{\gDomUB} \taub{i} \pb{i-1} = \floor{\frac{1}{2}\sum_{i \in \{i_1, \dots, i_m\}} 1}.
\end{align*}
Again, since $\taub{i} = 1$ if $i \in \{i_1, \dots, i_m\}$ and otherwise $\taub{i} = 0$, this is equivalent to 
\begin{align}\label{eqn:exact-alt-for-tau-delta}
\Sigma_2
&= \floor{\frac{1}{2}\sum_{i=1}^{\gDomUB} \taub{i}}.
\end{align}
We  now derive a lower bound on $\sum_{i=1}^{\gDomUB} \taub{i}$ similarly to how we derived \eqref{eqn:exact-alt-for-X}. 
By \cref{defn:tau_i} we have $\taub{i} = 1$ if and only if \begin{align}\label{aux:nexBound}
\gb(i) - \floor{\gb(i)} &\geq \frac{1}{2} + \left(\frac{16b}{i} - 4\right)^{-1/2}.
\end{align} 
We claim that this inequality holds  for all   integers $i$ such that 
\begin{align}\label{aux:304}
\floor{\gbinv(a + 1)} + 1 \leq i \leq \gbinv(a + \minfrac{a}),
\end{align}
for all integers $0 \leq a \leq \floor{\gb(1)} - 1 $, and 
where we recall the definition of $\minfrac{a}$ from the proposition statement. 
Applying $\gb$, which is decreasing, to~\eqref{aux:304} and using $\gbinv(a+1) <  \floor{\gbinv(a+1)}+1$, we obtain that  $$a+1 > \gb(i) \geq a + \minfrac{a}.$$ 
In particular, $a \geq \floor{\gb(i)}$, and thus,
\begin{align*}
\gb(i) - \floor{\gb(i)}  &\geq \minfrac{a}  
= \frac{1}{2} + \left(\frac{16b}{\gbinv(a + 1/2)} - 4\right)^{-1/2}.
\end{align*}
By assumption, we have $i \leq \gbinv(a + \minfrac{a})$. Since $\gbinv$ is decreasing and $\minfrac{a} \geq \frac{1}{2}$, this then implies that $i \leq \gbinv(a + \frac{1}{2})$. As $i$ is an integer, this  means  that $i \leq \floor{\gbinv(a + \frac{1}{2})}$. Finally, noting that $\left(\frac{16b}{t} - 4\right)^{-1/2}$ is increasing in $t$, we have 
that~\eqref{aux:nexBound} holds for all such  $i$. It follows that 
\begin{equation}\label{eqn:tight-Z-lower-bound}
    \sum_{i=1}^{\gDomUB} \taub{i} \geq  \sum_{a=0}^{\floor{\gb(1)} - 1} (\floor{\gbinv(a+ \minfrac{a})} - \floor{\gbinv(a+1)}).
\end{equation}   
Putting together~\eqref{eqn:smb-gdomub}, \eqref{eqn:exact-alt-for-X}, \eqref{eqn:exact-alt-for-tau-delta} and \eqref{eqn:tight-Z-lower-bound}, we obtain the lower bound
\begin{align}\label{eqn:lower-bound-on-smb-gdomub-ginvb}
    \smb{\gDomUB} & = \Sigma_1+\Sigma_2 \nonumber \\ &\geq \sum_{a=1}^{\floor{\gb(1)}} \floor{\gbinv(a)} + \floor{\frac{1}{2}\sum_{a=0}^{\floor{\gb(1)} - 1} \big(\floor{\gbinv(a+ \minfrac{a})} - \floor{\gbinv(a+1)} \big) }\nonumber \\ 
   & \geq -1 + \frac{1}{2} \sum_{a=0}^{\floor{\gb(1)}-1} (\floor{\gbinv(a+ \minfrac{a})} + \floor{\gbinv(a+1)}).
\end{align}
We derive a lower bound on $\floor{\gbinv(x)}$ for $0 \leq x \leq \gb(1)$. Recalling \cref{prop:g-is-invertible}, we have 
\begin{align*}
    \gbinv(x) &= 2b\left(1 + \frac{1}{4b} - \sqrt{\frac{(2x+1)^2}{(2x+1)^2+1}}\sqrt{1-\frac{(2x+1)^2+1}{16b^2}}\right) \\ 
    &\geq \frac{1}{2} +2b\left(1 - \sqrt{\frac{(2x+1)^2}{(2x+1)^2+1}}\right),
\end{align*}
and taking the floor yields 
\begin{equation}\label{eqn:lower-bound-on-ginvx}
    \floor{\gbinv(x)} \geq -\frac{1}{2} +2b\left(1 - \sqrt{\frac{(2x+1)^2}{(2x+1)^2+1}}\right).
\end{equation}
Putting together \eqref{eqn:lower-bound-on-smb-gdomub-ginvb} and  \eqref{eqn:lower-bound-on-ginvx} we have 
\begin{align*}
    \smb{\gDomUB} &\geq -1 + \frac{1}{2}\sum_{a=0}^{\floor{\gb(1)}-1} \left(-1 + 2b\left(2 - \sqrt{\frac{(2(a+ \minfrac{a})+1)^2}{(2(a+ \minfrac{a})+1)^2+1}} - \sqrt{\frac{(2a+3)^2}{(2a+3)^2+1}}\right)\right) \\ 
    &\geq b \, \Bigg(- \frac{2 + \gb(1)}{2b} + \sum_{a=0}^{\floor{\gb(1)}-1} \left(2 - \sqrt{\frac{(2(a+ \minfrac{a})+1)^2}{(2(a+ \minfrac{a})+1)^2+1}} - \sqrt{\frac{(2a+3)^2}{(2a+3)^2+1}}\right)\Bigg) \\
    &= C_bb. \tag*{\qedhere}
\end{align*}
\end{proof}
We show that $\Cb$ is increasing in $b$ which is needed when proving our result on the threshold bias $t(n,\mathcal{C})$. 
\begin{proposition}\label{prop:CB-increasing}
$\Cb$ is increasing in $b$.
\end{proposition}
\begin{proof}
We first check that $-\frac{2+\gb(1)}{2b}$, which is the first term in the definition of $\Cb$, is increasing in $b$. By \cref{defn:G}, we have $\gb(1) = \Gb(1) - \Gb(0) =  -\frac{1}{2}  + \sqrt{b - \frac{1}{4}}$. Thus, we have 
$$ -\frac{2+\gb(1)}{2b} = -\frac{3/2 + \sqrt{b - 1/4}}{2b},$$
which is easily seen to be increasing in $b$.
Next, we show that $\minfrac{a}$ is decreasing in $b$, for every fixed $a$ in $[0,  \floor{\gb(1)} - 1]$. Recalling \cref{prop:g-is-invertible}, we have
\begin{align*}
    \minfrac{a} = \frac{1}{2} +  \left(\frac{8}{1 + \frac{1}{4b} - \sqrt{\frac{(2a+2)^2}{(2a+2)^2+1}}\sqrt{1-\frac{(2a+2)^2+1}{16b^2}}} - 4\right)^{-1/2}.
\end{align*}
For any fixed $a$, it is clear that $\frac{(2a+2)^2+1}{16b^2}$ is decreasing in $b$. Thus, it follows that $1 + \frac{1}{4b} - \sqrt{\frac{(2a+2)^2}{(2a+2)^2+1}}\sqrt{1-\frac{(2a+2)^2+1}{16b^2}}$ is decreasing in $b$ so that $\minfrac{a}$ is decreasing in $b$. Hence, for all $a$ in $[0,  \floor{\gb(1)} - 1]$, the summands in the definition of $C_b$ are increasing in $b$. Moreover, these summands are positive and the upper limit of summation $\floor{\gb(1)} - 1$ is increasing in $b$. Thus, we conclude that $\Cb$ is increasing in $b$.
\end{proof}
Having determined the values of $b$ for which OBreaker is able to carry out their improved strategy, we are now able to prove results regarding the threshold bias.
\begin{theorem}\label{thm:main-result-on-threshold-bias}
Fix an integer $B \geq 3$. Then, for all sufficiently large integers $n$,
$$t(n,\mathcal{C}) \leq  \frac{n}{1+\CB} + O_B(1).$$
\end{theorem}
\begin{proof}
Let $n,b$ be positive integers and fix some integer $B \geq 3$. For $n < (1+\CB) B$, OBreaker wins for $b \geq (1+\CB) B$ via the trivial strategy. For $n \geq (1+\CB) B$ and $b \geq \frac{n}{1+\CB} + (1+C_B) B$ sufficiently large, we show that OBreaker has a winning strategy. If $n < b$ then OBreaker can again win via the trivial strategy so we may assume that $n \geq b$. Denote $b' := b - 1$ and $M':= \floor{g_{b'}^{-1}(1/2)}$, and recall the definition of $s_{b',M'}$. It can be computationally verified that $(1+C_3) 3 > 1$, and thus $(1+C_B) B > 1$ for $B \geq 3$. Therefore, $b' \geq \frac{n}{1+\CB}$. By \cref{prop:ineq-for-CB},  we have that $b' + s_{b',M'} \geq (1+C_{b'}) b'$. The assumption $n \geq (1+C_B) B$ implies $b' \geq \frac{n}{1+C_B} \geq B$. Since $C_B$ is increasing in $B$ by \cref{prop:CB-increasing}, we have $C_{b'} \geq C_B$ so that $b' + s_{b',M'}  \geq (1+\CB) b' \geq n.$ Putting this together, we have $b' \leq n \leq b' + s_{b',M'}$ which is required for applying \cref{prop:specialised-result-on-safety-transitions}. 

We will prove that, starting from an empty digraph, OBreaker can maintain that the digraph of the game remains safe if the game has not yet ended, and is otherwise an acyclic tournament. Since being safe entails being acyclic by \cref{prop:safe-implies-no-threats}, this proves that OBreaker has a winning strategy. More precisely, we prove by induction on the rounds that at the start of every round while the game has not yet ended, the digraph is $(i, s_{b', \myfloor{i}}, \delta_{b', \myfloor{i}})$-safe for some $i \in \Z_{\geq 0}$. Note that in this context, $\myfloor{i}$ naturally refers to $\min (i, M')$ rather than $\min (i, M)$. Letting $(A_{i}, B_{i})$ denote the UDB witnessing the $(i, s_{b', \myfloor{i}}, \delta_{b', \myfloor{i}})$-safety of $D$, we also show that $A_{i} \cup  B_{i} \neq V$ for $1 \leq i \leq M'$ which is required for applying  \cref{prop:specialised-result-on-safety-transitions}.

At the start of the game, the digraph is empty and this is indeed  $(0,s_{b', 0},\delta_{b',0})$-safe and is trivially witnessed by the empty UDB $(\emptyset, \emptyset)$. 

For the induction step, we show that if the game is not over and the digraph is $(i-1, s_{b', \myfloor{i-1}}, \delta_{b', \myfloor{i-1}})$-safe for some $i \in \Z_{>0}$ at the start of some round, then the digraph is $(j, s_{b', \myfloor{j}}, \delta_{b', \myfloor{j}})$-safe for some $j \in \Z_{\geq 0}$ when the round ends. Assuming that $A_{i-1} \cup  B_{i-1} \neq V$ if $i \leq M'$, we also show that $A_j \cup B_j \neq V$ if $j \leq M'$. Let $D$ be the $(i-1, s_{b', \myfloor{i-1}}, \delta_{b', \myfloor{i-1}})$-safe digraph and let $e$ be the edge directed by OMaker in the current round. If $\mathcal{A}(D + e)$ is empty, then the final tournament is acyclic by \cref{prop:safe-implies-no-threats}. Otherwise, OBreaker applies \cref{prop:specialised-result-on-safety-transitions} to $D$ and $e$ and directs the obtained set of arcs $S \subseteq \mathcal{A}(D + e)$ so that $(D+e) \cup S$ is $(i, s_{b', \myfloor{i}}, \delta_{b', \myfloor{i}})$-safe and $|S| \leq b'$. Observe that \cref{prop:specialised-result-on-safety-transitions} ensures that if $i \leq M'$, then $A_i \cup B_i \neq V$. If $|S| > 0$, then OBreaker ends their current turn. Otherwise, suppose that  $|S| = 0$. In this case, OBreaker then directs an arbitrary available arc $e' \in \mathcal{A}(D + e)$. If $\mathcal{A}(D \cup \{e, e'\})$ is empty, then the final tournament is acyclic by \cref{prop:safe-implies-no-threats}. Otherwise, OBreaker applies \cref{prop:specialised-result-on-safety-transitions} on $D + e$ and $e'$ and directs the obtained set of arcs $S' \subseteq \mathcal{A}(D \cup \{e, e'\})$ so that $(D \cup \{e, e'\}) \cup S'$ is $(i+1, s_{b', \myfloor{i+1}}, \delta_{b', \myfloor{i+1}})$-safe and $|S'| \leq b'$. Here, \cref{prop:specialised-result-on-safety-transitions} ensures that if $i + 1 \leq M'$, then $A_{i+1} \cup B_{i+1} \neq V$. In all cases, OBreaker directs at least one edge and no more than $b' + 1 = b$ edges. Thus, either the game ends and the final digraph is an acyclic tournament, or OBreaker is able to ensure that after their move, the digraph is $(j, s_{b', \myfloor{j}}, \delta_{b', \myfloor{j}})$-safe for some $j \in \Z_{\geq 0}$. Moreover, we have that $A_j \cup B_j \neq V$ if $j \leq M'$.

Hence, OBreaker has a winning strategy in all cases for $b \geq  \frac{n}{1+C_B} + (1+C_B) B$ which implies that $t(n,\mathcal{C}) \leq  \frac{n}{1+\CB} + O_B(1)$.\qedhere
\end{proof}
To obtain \cref{thm:main-result-on-threshold}, we can apply \cref{thm:main-result-on-threshold-bias} with a fixed value of $B$. 
\mainresultonthreshold*
\begin{proof}
This is immediate from \cref{thm:main-result-on-threshold-bias} by setting $B=10^9$ and noting that
$\frac{1}{1 + C_{10^9}} < 0.7841$ by evaluation.
\end{proof}
Applying \cref{thm:main-result-on-threshold-bias} with a larger value of $B$ would, in principle, yield a tighter upper bound on $t(n,\mathcal{C})$. However, in practice, after evaluating $\frac{1}{1+C_B}$ with larger values of $B$, we suspect that $\lim_{B \to \infty} \frac{1}{1+C_B} > 0.784$ so that one cannot actually obtain a result which is materially better in this manner.

\section{Concluding Remarks}
We achieve a significantly tighter upper bound on $t(n,\mathcal{C})$ with our new strategy. 
Yet, we believe that there is slightly more room for improvement. 
Recall that our strategy is based on maintaining safety via \cref{prop:main-result-on-safety-transitions}. Choosing $x_i$ to be $\floor{\gb (i)}$ enabled us to achieve $(n-b)$-safety quickly while directing at most $b$ edges in each round, and while still being able to effectively lower bound the size of the UDB. A better strategy would be to instead choose $x_i$  to be as large as possible in each round. Here, obtaining explicit expressions for the size of the UDB, as well as for $t(n,\mathcal{C})$ seems to be challenging. Computational experiments (with $b = 10^6$) suggest that the upper bound on the threshold bias can be improved to $t(n,\mathcal{C}) \leq 0.7692n + O(1)$.

We believe that the optimal strategy for OBreaker involves maintaining safety, but not necessarily via \cref{prop:main-result-on-safety-transitions}. Using \cref{prop:main-result-on-safety-transitions} necessitates that $|A| - k$ and $|B| - \ell$ differ by at most one, but an optimal strategy likely requires no such constraint. For instance, if OMaker is only directing edges into $A$ or only directing edges out from $B$, which we believe to be their best response to OBreaker's strategy of building a safe digraph, then it may be in OBreaker's interest to more aggressively add more vertices to one of $A$ or $B$ in response. However, we saw that in practice this seemed to only result in a relatively minor improvement over the aforementioned stronger strategy. This leads us to the conjecture mentioned in the introduction.
\begin{conjecture}
    $t(n,\mathcal{C}) \geq 3n/4 - O(1)$.
\end{conjecture}

As a final remark, recall that OMaker's best known strategy involves extending the longest path and was shown to be winning for $b \leq n/2 - 2$. We believe that such a strategy is close to optimal, and a more refined proof can show that this strategy is winning for OMaker for significantly larger values of $b$.

\bibliographystyle{plain}
\bibliography{references}
\end{document}